\documentclass[12pt,a4paper]{article}    

\usepackage{microtype}
\usepackage[margin=1in]{geometry} 
\usepackage{amsmath,amssymb,amsthm}
\usepackage{color}
\usepackage{relsize}
\usepackage[perpage,symbol*]{footmisc}
\usepackage{authblk}
\usepackage{enumerate}
\usepackage{cases}
\usepackage{cite}
\usepackage[none]{hyphenat}
\usepackage{setspace}
\usepackage{enumitem}
\allowdisplaybreaks

\newtheorem{theorem}{Theorem}[section]

\newtheorem{lemma}[theorem]{Lemma}

\newtheorem{remark}{Remark}

\makeatletter
\renewenvironment{proof}[1][\proofname]{\par
	\pushQED{\qed}%
	\normalfont \topsep6\p@\@plus6\p@\relax
	\trivlist
	\item[\hskip\labelsep
	\itshape #1\@addpunct{.}]\ignorespaces
}{%
	\popQED\endtrivlist\@endpefalse
}
\makeatother

\usepackage{authblk}

\newcommand{\be}{\begin{equation}}
	 \newcommand{\ee}{\end{equation}}
\newcommand{\bea}{\begin{eqnarray}}
	 \newcommand{\eea}{\end{eqnarray}}
\newcommand{\e}{{\rm e}}

\newcommand{\keywords}[1]{\textbf{Keywords:} #1}

\newtheorem{exm}{Example}

\newtheorem{assumption}{Assumption}

\def\tr {{\rm tr}}

\numberwithin{equation}{section}

\title{Ladder Operators for Laguerre-type and Jacobi-type\\ Orthogonal Polynomials}

\author[1,*]{Shulin Lyu}
\author[1,$\dag$]{Yuanfei Lyu}  
\affil[1]{School of Mathematics and Statistics, Qilu University of Technology (Shandong Academy of Sciences), Jinan 250353, China}
\affil[*]{\texttt{lvshulin1989@163.com}, Corresponding author}
\affil[$\dag$]{\texttt{lvyuanfei5720@163.com}}
\date{\today}

\begin{document}
	\maketitle
	
	\begin{abstract}
In the literature concerning the Laguerre-type weight function $x^\lambda w_0(x), x\in[0,+\infty)$, the Jacobi-type weight function $(1-x)^{\alpha}(1+x)^{\beta}w_0(x),x\in[-1,1]$, and the shifted Jacobi-type weight function $x^{\alpha}(1-x)^{\beta}w_0(x), x\in[0,1]$, with $w_0(x)$ positive and continuously differentiable, the parameters $\lambda,\alpha,\beta$ are usually constrained  to be strictly positive to ensure the validity of the results. Recently, in [C. Min and P. Fang, Physica D 473 (2025), 134560 (9pp)], the ladder operators for the monic Laguerre-type orthogonal polynomials with $\lambda>-1$ were derived by exploiting the orthogonality properties. The quantities $A_n$ and $B_n$, which appear as coefficients in the ladder operators, exhibit different expressions compared with the previous ones for $\lambda>0$. In this paper, we construct an alternative deduction by making use of the Riemann-Hilbert problem satisfied by the orthogonal polynomials. Moreover, we employ both derivation strategies mentioned above to produce the ladder operators for the monic standard and shifted Jacobi-type orthogonal polynomials with $\alpha,\beta>-1$.
When $\lambda,\alpha,\beta$ are restricted to positive values, our expressions of $A_n$ and $B_n$ are consistent with those in prior work. We present examples to validate our findings and generalize the existing conclusions, established by using the three compatibility conditions of the ladder operators and differentiating the orthogonality relations for the monic orthogonal polynomials, from $\lambda,\alpha,\beta>0$ to $\lambda,\alpha,\beta>-1$.
	\end{abstract}
	\noindent
	


\section{Introduction}\label{section1}
We consider the Laguerre-type weight function
\begin{equation}\label{wL}
	w_{_L}(x)=x^\lambda w_0(x),\qquad x\in [0, +\infty),\quad\lambda>-1,
\end{equation}
the Jacobi-type weight function
\begin{equation}\label{wJ}
w_{_J}(x)=(1-x)^{\alpha}(1+x)^{\beta}w_0(x),\qquad x\in[-1,1],\quad\alpha,\beta>-1,
\end{equation}
and the shifted Jacobi-type weight function
\begin{equation}\label{wJS}
w_{_{JS}}(x)=x^{\alpha}(1-x)^{\beta}w_0(x),\qquad x\in[0,1],\quad\alpha,\beta>-1.
\end{equation}
Denote the above three types of weight functions by $w(x),x\in I$, where $I=[0,+\infty),[-1,1]$ and $[0,1]$ for $w_{_L}(x),w_{_J}(x)$ and $w_{_{JS}}(x)$ respectively. We make the following assumptions:
\begin{assumption}\label{ass:moments}
All moments of $w(x)$,  i.e. $\int_Ix^jw(x)dx,j=0,1,\ldots$, are finite;
\end{assumption}
\begin{assumption}\label{ass:w0}
$w_0(x)$ is a continuously differentiable function, $w_0(x)>0$ and $w_0(x)$ are bounded at the finite endpoints of $I$;
\end{assumption}
\begin{assumption}\label{ass:wL}
$\lim\limits_{x\to +\infty}w_{_L}(x)\pi(x)=0$ with $\pi(x)$ representing an arbitrary polynomial.
\end{assumption}

For $w(x)$, there exists a unique sequence of monic orthogonal polynomials \cite[Theorem 2.1.1]{Ismail}
\begin{align}\label{mp}
P_j(x):=x^j+\textbf{p}(j)x^{j-1}+\dots+P_j(0),\quad\qquad j=1,2,\cdots,
\end{align}
with $P_0(x):=1$ and $\textbf{p}(0):=0$, satisfying
\begin{align}\label{dp}	 \int_{I}P_i(x)P_j(x)w(x)dx=h_j\delta_{ij},\quad\qquad i,j=0,1,\cdots,
\end{align}
where $\delta_{ij}$ is the Kronecker delta function, i.e. $\delta_{ij}=1$ for $i=j$ and $0$ otherwise. Since every polynomial $\pi_k(x)$ of degree $k$ can be expressed as a linear combination of $\{P_{\ell}(x)\}_{\ell=0}^{k}$, one has by \eqref{dp}
\begin{align}\label{or-1}
\int_I\pi_k(x)P_{j}(x)w(x)dx=0, \quad\qquad 0\leq k\leq j-1,
\end{align}
for $j\geq1$, and
\begin{align}\label{or-2}
\int_I\pi_j(x)P_{j}(x)w(x)dx=\gamma_jh_{j},\quad\qquad j\geq0,
\end{align}
with $\gamma_j$ denoting the leading coefficient of $\pi_j(x)$.
The Hankel determinant generated by the moments of $w(x)$ reads
\begin{align*}\label{D_nh_n}
D_n[w(x)]:=&\det \left( \int_Ix^{i+j}w(x)dx\right) _{i,j=0}^{n-1}.
\end{align*}
According to Heine's formula, $D_n$ admits the following alternative representations \cite[Section 2.1]{Ismail}
\begin{align*}
D_n[w(x)]=&\frac{1}{n!}\int_{I^n}\prod_{1\leq i<j\leq n}(x_i-x_j)^2\prod_{k=1}^n w(x_k)dx_1dx_2\cdots dx_n\\
=&\prod_{j=0}^{n-1}h_j.
\end{align*}

From \eqref{mp}-\eqref{dp}, one obtains the three-term recurrence relation
\begin{equation}\label{3r}	 xP_j(x)=P_{j+1}(x)+\alpha_jP_j(x)+\beta_jP_{j-1}(x),\quad\qquad j\geq0,
\end{equation}
with the initial condition $P_{-1}(x):=0$, and the coefficients are given by
\begin{align}\label{alp}
\alpha_j=\textbf{p}(j)-\textbf{p}(j+1),\quad\qquad j\geq0,
\end{align}
\begin{equation}\label{bth}
	\beta_j=\frac{h_j}{h_{j-1}},\quad\qquad j\geq1.
\end{equation}
Using \eqref{3r}, one gets the Christoffel-Darboux formula
\begin{equation}\label{cd}	 \sum_{j=0}^{n-1}\frac{P_j(x)P_j(y)}{h_j}=\frac{P_n(x)P_{n-1}(y)-P_{n-1}(x)P_n(y)}{h_{n-1}(x-y)}.
\end{equation}
Combining \eqref{dp}-\eqref{cd}, under the assumption that $\lambda,\alpha,\beta>0$, one derives a pair of ladder operators for $\{P_n(z)\}$:
	\begin{equation*}
P'_n(z)=-B_n(z)P_n(z)+\beta_nA_n(z)P_{n-1}(z),
	\end{equation*}
	\begin{equation*}
P_{n-1}'(z)=\left(B_n(z)+v'(z)\right)P_{n-1}(z)-A_{n-1}(z)P_n(z),
	\end{equation*}
with $n\geq0$, where $A_n(z)$ and $B_n(z)$ are defined by
	\begin{align}\label{a-n}		 A_n(z):=\frac{1}{h_n}\int_I\frac{v'(z)-v'(x)}{z-x}P_n^2(x)w(x)dx,\quad\qquad n\geq0,
	\end{align}
	\begin{equation}\label{b-n}		 B_n(z):=\frac{1}{h_{n-1}}\int_I\frac{v'(z)-v'(x)}{z-x}P_{n-1}(x)P_n(x)w(x)dx,\quad\qquad n\geq1,
\end{equation}
with $A_{-1}(z)=B_0(z):=0$. Here $v(x):=-\ln w(x)$. Moreover, $A_n(z)$ and $B_n(z)$ satisfy the following three identities:
 \begin{align}
B_{n+1}(z)+B_n(z)&=(z-\alpha_n)A_n(z)-v'(z)\tag{$S_1$},\\
1+(z-\alpha_n)(B_{n+1}(z)-B_n(z))&=\beta_{n+1}A_{n+1}(z)-\beta_nA_{n-1}(z)\tag{$S_2$},			 \end{align}
for $n\geq0$, and
\begin{align}
B^2_n(z)+v'(z)B_n(z)+\sum_{j=0}^{n-1}A_j(z)&=\beta_nA_n(z)A_{n-1}(z)\tag{$S^{'}_2$},
\end{align}
for $n\geq1$. The above-mentioned results hold for a generic weight function $w(x)={\rm e}^{-v(x)}$ which has finite moments of all orders on $[a,b]$ with $w(a)=w(b)=0$, where $v(x)$ is differentiable and $v'(x)$ is Lipschitz continuous. See \cite{Magnus95,ChenIsmail97,WvA18} for more details.

To establish ODEs or PDEs to characterize Hankel determinants, researchers make use of $(S_1),(S_2), (S_2')$ and the orthogonality relation \eqref{dp}: By computing $A_n$ and $B_n$ where auxiliary quantities are usually introduced, and substituting them into $(S_1)$, $(S_2)$ and $(S_2')$, one expresses the recurrence coefficients and $\textbf{p}(n)$ in terms of the auxiliary quantities which are shown to satisfy a system of difference equations that can be iterated in $n$; by differentiating the orthogonality condition (i.e. \eqref{dp} for our problem) w.r.t. the variable introduced in the weight function, one gets differential relations; combining all the results together, one derives Riccati equations for the auxiliary quantities from which ODEs or PDEs follow; with the help of the relation $D_n=\prod_{j=0}^{n-1}h_j$, one connects the logarithmic derivative of the Hankel determinant $D_n$ with the auxiliary quantities and $\textbf{p}(n)$, and finally establishes ODE or PDE for it.

We call the above derivation scheme the ladder operator approach. As a powerful research tool, it has been widely applied to study quantities arising from Hermitian random matrices- particularly those expressible as Hankel determinants (often up to a multiplicative constant), for example, partition functions, gap probabilities, moment generating functions of linear statistics. Among these studies, the Laguerre-type unitary ensembles defined by \eqref{wL} with $\lambda>0$, and the standard and shifted Jacobi-type unitary ensembles governed by \eqref{wJ}-\eqref{wJS} with $\alpha,\beta>0$ have drawn considerable attention from researchers, for example,
\begin{itemize}
\item [(i)] the semi-classical Laguerre  and Jacobi weights with $w_0(x)=(x-t)^{\gamma}$ \cite{ChenMcKay12,DZ10} and $w_0(x)={\rm e}^{-tx}$ \cite{BCE10};
\item [(ii)] the singularly perturbed Laguerre and Jacobi weights with $w_0(x)={\rm e}^{-t/x}$ \cite{ChenIts2010,CD10};
\item [(iii)] the classical Jacobi weight $(1-x)^{\alpha}(1+x)^{\beta}$ \cite{ChenIsmail05} and the symmetric Jacobi weights \cite{BCH15,MinChen22,MinChenNPB20,MinChen21}.
\end{itemize}
The reason why $\lambda,\alpha,\beta$ were restricted to positive values in these works is that the derivation of the expressions for $A_n$ and $B_n$ requires the weight function to vanish at the endpoints of the domain.

Recently, by utilizing the orthogonality relation \eqref{dp} and the Christoffel-Darboux formula \eqref{cd}, Min and Fang \cite{MinFang25} refined the derivation of the ladder operators for the Laguerre-type orthogonal polynomials from $\lambda>0$ to $\lambda>-1$, and the three compatibility conditions $(S_1),(S_2),$ and $(S_2')$ were proven to still hold. In this paper, we present an alternative deduction by making use of the Riemann-Hilbert problem satisfied by the orthogonal polynomials. Moreover, we employ these two methods to formulate the ladder operators for the standard and shifted Jacobi-type orthogonal polynomials with $\alpha,\beta>-1$. We also present the ladder operators for these three families of orthogonal polynomials exhibiting Fisher-Hartwig singularities of jump or root type. In addition, when $\lambda,\alpha,\beta$ are constrained to be positive, our results align with those in previous work. Through representative examples, we verify our findings and generalize the existing conclusions established by using the ladder operator approach from $\lambda,\alpha,\beta>0$ to $\lambda,\alpha,\beta>-1$.

This paper is organized as follows. Section 2 addresses the Riemann-Hilbert problem for monic orthogonal polynomials and the properties of the Cauchy transform. These tools are then applied in Sections 3-5 to derive ladder operators for the Laguerre-type, Jacobi-type, and shifted Jacobi-type orthogonal polynomials. Illustrative examples are presented to validate our results and demonstrate the generalization of the conclusions in prior studies. In addition, with the help of properties \eqref{3r}-\eqref{cd}, we provide in Section 4 an alternative proof of the ladder operators for Jacobi-type orthogonal polynomials, by expressing $(1-z^2)P_n'(z)$ as a linear combination of $\{P_k(z)\}_{k=1}^{n+1}$.

\section{Riemann-Hilbert Problem and Cauchy Transform}
In this section, we present the Riemann-Hilbert problem (RHP) for the monic orthogonal polynomials associated with \eqref{wL}-\eqref{wJS}, and state the properties of the Cauchy transform.

Define the $2\times2$ matrix-valued function $Y(z):\mathbb{C} \rightarrow \mathbb{C} \setminus I$ by \begin{equation}\label{defY}
	Y(z) := \begin{pmatrix}
		P_n(z) & \frac{1}{2\pi i} \displaystyle\int_{I}\frac{P_n(x)w(x)}{x-z} dx \\
		\frac{-2\pi i}{h_{n-1}} P_{n-1}(z) & \frac{-1}{h_{n-1}}\displaystyle\int_{I}\frac{P_{n-1}(x)w(x)}{x-z}dx \end{pmatrix},
\end{equation}
where the weight function $w(x)=w_{_L}(x),w_{_J}(x)$ and $w_{_{JS}}(x)$ with $I=[0,+\infty),[-1,1]$ and $[0,1]$ respectively, and $\{P_n(x)\}$ are the associated monic orthogonal polynomials determined by \eqref{mp}-\eqref{dp}.
According to the standard theory in \cite{FIK}, we know that $Y(z)$ satisfies the following three conditions:	
\begin{itemize}
	\item [(i)]$Y$ is analytic on
	$\mathbb{C} \setminus I$;
	\item[(ii)] $Y$ satisfies the jump condition
	\[Y_+(x)=Y_-(x)\left( \begin{array}{cc}
		1 & w(x) \\
		0 & 1
	\end{array} \right),\qquad x\in I^{\circ},\]
where $I^{\circ}=(0,+\infty),(-1,1)$ and $(0,1)$ for $w(x)=w_{_L}(x),w_{_J}(x)$ and $w_{_{JS}}(x)$\\ respectively;
	\item[(iii)] $Y$ has the asymptotic behavior as $z\rightarrow \infty$
	$$Y(z)=\left( I+O\left( \frac{1}{z}\right) \right) \left( \begin{array}{cc}
		z^n & 0 \\
		0 & z^{-n}
	\end{array} \right).$$
\end{itemize}
The above criteria define a standard RHP, whereas a concrete weight function necessitates supplementary hypotheses to guarantee that $Y(z)$ is the unique solution of the RHP. Refer to \cite[Sections 22.4-22.5]{Ismail} for the RHPs induced by the Laguerre weight $x^{\lambda}\e^{-x}$ and the Jacobi weight $(1-x)^{\alpha}(1+x)^{\beta}$, with $\{P_n(x)\}$ denoting the monic Laguerre and Jacobi polynomials respectively. See also \cite[Section 3.2]{Deift} and \cite{ChenIts2010,WuXu21, WangFan}.

Denote the Cauchy transform of a function $f$ defined on $I$ by
	\begin{equation}\label{defC}		 C(f):=\int_{I}\frac{f(x)}{x-z}dx.
		\end{equation}
The following property, stated in \cite[p.55-p.56]{WvA18} without a proof, will be crucial in our subsequent derivations.
\begin{lemma}
If
\begin{align}\label{cond}
\int_{I}x^{k}f(x)dx=0,\quad\qquad k=0,1,\cdots,n-1,
\end{align}
then, for every polynomial $Q_{m}$ of degree $m (\le n)$, one has
	\begin{equation}\label{cpro}
	Q_{m}(z)C(f)=C(f Q_{m}),\quad\qquad 0\leq m\leq n,
	\end{equation}
so that
\begin{equation}\label{Cpn}
Q_{m}(z)C(P_nw)=C(Q_mP_{n}w),\quad\qquad 0\leq m\leq n.
\end{equation}

\end{lemma}

\begin{proof}
By definition \eqref{defC}, we have
\begin{align}\label{C1}
C(f Q_{m})-Q_{m}(z)C(f)=\int_{I}\frac{Q_m(x)-Q_m(z)}{x-z}f(x)dx.
\end{align}
Writing
 \[Q_m(y)=a_my^m+a_{m-1}y^{m-1}+\cdots+a_1y+Q_m(0),\qquad a_m\neq0,\]
we find, for $m\geq3$,
\begin{align*}
Q_m(x)-Q_m(z)=&a_m\left(x^m-z^m\right)+a_{m-1}(x^{m-1}-z^{m-1})+\cdots+a_1(x-z)\\
=&(x-z)\left(a_mx^{m-1}+(a_mz+a_{m-1})x^{m-2}+\cdots+a_1\right),
\end{align*}
where the second equality is due to
\[y^k-z^k=(y-z)(y^{k-1}+y^{k-2}z+\cdots+yz^{k-2}+z^{k-1}),\qquad k=3,4,\cdots.\]
It follows that
\[\frac{Q_m(x)-Q_m(z)}{x-z}=a_mx^{m-1}+(a_mz+a_{m-1})x^{m-2}+\cdots+a_1.\]
Plugging it back into \eqref{C1}, in view of \eqref{cond}, we are led to \eqref{cpro}. We readily see that \eqref{cpro} is also valid for $m=0,1,2$.

From \eqref{or-1}, we get
$\int_I x^kP_n(x)w(x)dx=0$ for $k=0,1,\cdots,n-1.$
Hence, according to \eqref{cond}-\eqref{cpro}, we come to \eqref{Cpn}.
\end{proof}

With the notation \eqref{defC}, $Y(z)$ defined by \eqref{defY} can be written as
\begin{equation}\label{ymatrix}
	Y(z) = \begin{pmatrix}
		P_n(z) & \frac{1}{2\pi i} C(P_nw) \\
		\frac{-2\pi i}{h_{n-1}} P_{n-1}(z) & \frac{-1}{h_{n-1}}C(P_{n-1}w)  \\
\end{pmatrix}.
\end{equation}
According to the RHP satisfied by $Y(z)$, one gets $\det Y \equiv 1$ \cite[Section 4.2]{WvA18}, which yields
\begin{align}\label{I1}
-P_n(z)C(P_{n-1}w)+P_{n-1}(z)C(P_nw) =h_{n-1},
\end{align}
and
\begin{equation}
	\label{y-1}
	Y^{-1}(z) = \begin{pmatrix}		 \frac{-1}{h_{n-1}}C(P_{n-1}w) & -\frac{1}{2\pi i} C(P_nw)\\
\frac{2\pi i}{h_{n-1}} P_{n-1}(z)  & P_n(z)  \\
\end{pmatrix}.	
\end{equation}
Define $R(z):=Y'(z)\cdot Y^{-1}(z)$. It follows from \eqref{ymatrix} and \eqref{y-1} that
\begin{align}\label{rmatrixc}
R(z)=\begin{pmatrix}			 \frac{1}{h_{n-1}}(-P'_nC(P_{n-1}w)+P_{n-1}C(P_nw)') & \frac{1}{2\pi i}(-P'_nC(P_nw)+P_nC(P_nw)')  \\
\frac{2\pi i}{h_{n-1}^{2}}(P'_{n-1}C(P_{n-1}w)-P_{n-1}C(P_{n-1}w)') & \frac{1}{h_{n-1}}(P'_{n-1}C(P_nw)-P_nC(P_{n-1}w)')
\end{pmatrix},
\end{align}
where $'$ denotes $d/dz$. Moreover, since $Y'(z)=R(z)\cdot Y(z)$, we have
	\begin{equation}\label{p'nyr}		 P'_n(z)=\left(Y'(z)\right)_{1,1}=R_{1,1}P_n(z)-\frac{2\pi i}{h_{n-1}}R_{1,2}P_{n-1}(z),
	\end{equation}
	\begin{equation}\label{pn-1'y}		 P_{n-1}'(z)=-\frac{h_{n-1}}{2\pi i}\left(Y'(z)\right)_{2,1}=-\frac{h_{n-1}}{2\pi i}R_{2,1}P_n(z)+R_{2,2}P_{n-1}(z),
	\end{equation}
where $R_{i,j}:=\left(R(z)\right)_{i,j}$ for $i,j=1,2$. Using \eqref{rmatrixc} combined with \eqref{Cpn} to derive the expressions for $\{R_{1,1}, R_{1,2}, R_{2,1},R_{2,2}\}$, and substituting them into \eqref{p'nyr}-\eqref{pn-1'y}, we finally come to the ladder operators satisfied by $P_n(z)$.

The approach outlined above was employed in \cite{WvA18} to establish the ladder operators for monic orthogonal polynomials associated with a weight function that vanishes at the endpoints of its domain. Obviously, their derivation is not applicable to our problem. In addition, in \cite{ChenIts2010}, the ladder operators for the monic polynomials orthogonal w.r.t. $x^{\alpha}\e^{-x-s/x}, \alpha,s>0$ was deduced from the corresponding RHP. The authors transformed $Y(z)$ into $\Psi(z)$ to achieve a constant jump matrix for the latter. By using the behavior of $\Psi(z)$ as $z\rightarrow0$ and $z\rightarrow\infty$, the Lax pair $A(z):=\Psi'(z)\Psi^{-1}(z)$ was obtained, from which the ladder operators followed. Their method is clearly different from ours, since we do not need to make transformation or consider the specific form of $Y(z)$ as $z\rightarrow\infty$. Actually, in this paper, the behavior of $Y(z)$ at $\infty$, given by Condition (iii) of the RHP, is used only to establish $\det Y(z)\equiv1$ which is essential to derive \eqref{I1}-\eqref{y-1}.

\section{Ladder Operators for Laguerre-type Orthogonal Polynomials with $\lambda>-1$}
It was shown in \cite{MinFang25} that the monic Laguerre-type polynomials orthogonal w.r.t. \eqref{wL}, i.e. $w_{_L}(x)=x^\lambda w_0(x),~ x\in [0, +\infty),~ \lambda>-1,$ satisfy a pair of ladder operators, and the compatibility conditions $(S_1), (S_2)$ and $(S_2')$ presented in Introduction are still valid. We restate the results in the following theorem.
\begin{theorem}\label{main}
The ladder operators below hold for the monic Laguerre-type orthogonal polynomials:
\begin{align}
P_n'(z)=&-B_n(z)P_n(z)+\beta_nA_n(z)P_{n-1}(z),\label{lo}\\
P_{n-1}'(z)=&\left(B_n(z)+v_{_L}'(z)\right)P_{n-1}(z)-A_{n-1}(z)P_n(z),\label{ro}
\end{align}
with $n\geq0$, where $A_n(z)$ and $B_n(z)$ are defined by
     \begin{align}\label{AN}	 A_n(z):=\frac{1}{z}\cdot\frac{1}{h_{n}}\int_{0}^{+\infty}\frac{zv_{_L}'(z)-xv_{_L}'(x)}{z-x}P_n^{2}(x)w_{_L}(x)dx,\qquad n\geq0,
     \end{align}
     \begin{align}\label{BN}
      B_n(z):=\frac{1}{z}\left(\frac{1}{h_{n-1}}\int_{0}^{+\infty}\frac{zv_{_L}'(z)-xv_{_L}'(x)}{z-x}P_{n-1}(x)P_n(x)w_{_L}(x)dx
     -n\right),\qquad n\geq1,
     \end{align}
with $A_{-1}(z)=B_0(z):=0$. Here $v_{_L}(x):=-\ln w_{_L}(x)$ with $w_{_L}(x)$ given by \eqref{wL}. Moreover, $A_n(z)$ and $B_n(z)$ satisfy $(S_1), (S_2)$ and $(S_2')$.
\end{theorem}
      	
In \cite{MinFang25}, the lowering operator \eqref{lo} was established by expressing $zP_n'(z)$ in terms of $\{P_j\}_{j=0}^n$, with the help of the definition and properties of $\{P_j\}$, namely, \eqref{mp}-\eqref{cd}. By using the definitions of $A_n$ and $B_n$, in view of the recurrence relation \eqref{3r}, the compatibility conditions $(S_1)$ and $(S_2)$ were obtained, a combination of which yielded $(S_2')$. The raising operator \eqref{ro} was a result of $(S_1)$ and the lowering operator.

\begin{remark}
$A_n$ defined by \eqref{AN} can also be written as
\begin{align}		 A_n(z)=&\frac{1}{h_n}\int_{0}^{+\infty}\frac{v_{_L}'(z)-\frac{x}{z}v_{_L}'(x)}{z-x}P_n^{2}(x)w_{_L}(x)dx\nonumber\\		 =&\frac{1}{h_n}\int_{0}^{+\infty}\frac{v_{_L}'(z)-v_{_L}'(x)}{z-x}P_n^{2}(x)w_{_L}(x)dx+\frac{1}{z}\cdot \frac{1}{h_n}\int_{0}^{+\infty}P_n^{2}(x)w_{_L}(x)v_{_L}'(x)dx.\label{An-1}
\end{align}
Noting that $w_{_L}(x)v_{_L}'(x)dx=-dw_{_L}(x)$ and $w_{_L}(0)=w_{_L}(+\infty)=0$ when $\lambda>0$, through integration by parts, we find for $\lambda>0$
\begin{align*}
\frac{1}{h_n}\int_{0}^{+\infty}P_n^{2}(x)w_{_L}(x)v_{_L}'(x)dx
=&-\left.\frac{P_n^{2}(x)}{h_n}w_{_L}(x)\right|_{x=0}^{x=+\infty}+\frac{1}{h_n}\int_0^{+\infty}2 P_n'(x)P_n(x)w_{_L}(x)dx\\
=&0,
\end{align*}
where we make use of Assumption \ref{ass:wL} given in Introduction and \eqref{or-1} to get the second identity. Consequently, \eqref{An-1} {\emph(}equivalently, \eqref{AN}{\emph)} reduces to \eqref{a-n}. Similarly, we can show that, when $\lambda>0$, $B_n$ given by \eqref{BN} becomes \eqref{b-n}. Therefore, we conclude that, when $\lambda>0$, $A_n$ and $B_n$ defined by \eqref{AN}-\eqref{BN} are consistent with \eqref{a-n}-\eqref{b-n}. This consistency was also pointed out in \cite[Remark 1]{MinFang25}.

{\sloppy The above conclusion indicates that when $\lambda>0$ we can make use of either \eqref{AN}-\eqref{BN} or \eqref{a-n}-\eqref{b-n} to compute $A_n$ and $B_n$, while for $-1<\lambda\leq0$ we can only use \eqref{AN}-\eqref{BN}. Nevertheless, we observe that in previous work where $\lambda>0$, after $A_n$ and $B_n$ were computed by using \eqref{a-n}-\eqref{b-n}, an integration by parts was usually applied. This operation leads to the fact that the obtained expressions are identical to those deduced by using \eqref{AN}-\eqref{BN} which are valid for $\lambda>-1$. Refer to \eqref{An-1}. Consequently, the results established by substituting the expressions of $A_n$ and $B_n$ into $(S_1), (S_2), (S_2')$ in those studies for $\lambda>0$ can be generalized to $\lambda>-1$. The above analysis will be validated in Section \ref{exL}.}
\end{remark}

In this section, we will present an alternative derivation of Theorem \ref{main} and use it in several examples to extend the validity of  existing results from $\lambda>0$ to $\lambda>-1$.
	
\subsection{Proof of the ladder operators}
We apply \eqref{rmatrixc}-\eqref{pn-1'y} to deduce the lowering and raising operators, with the aid of \eqref{Cpn}. Combing them together produces $(S_1)$ and $(S_2)$ which further lead to $(S_2')$.

From \eqref{p'nyr}-\eqref{pn-1'y}, we know that the expressions of the four elements of $R(z)$ are essential to derive the ladder operators. Hence, we need the following lemma.
\begin{lemma}\label{lemmal}
$C(P_{n-1}P_nw_{_L})'$  and  $ C(P_{n}^{2}w_{_L})'$ can be expressed as follows
\begin{equation}\label{lemma21}	
\begin{aligned}		 C(P_{n-1}P_{n}w_{_L})'=&C(P_{n-1}'P_nw_{_L})+C(P_{n-1}P_n'w_{_L})\\			 &+\frac{1}{z}\left(-\int_{0}^{+\infty}\frac{xv_{_L}'(x)}{x-z}P_{n-1}(x)P_{n}(x)w_{_L}(x)dx+n\cdot h_{n-1}\right),
\end{aligned}
\end{equation}
 \begin{align}\label{CPn2'}
C(P_{n}^{2}w_{_L})'	 =&2C(P'_nP_nw_{_L})-\frac{1}{z}\int_{0}^{+\infty}\frac{xv'_{_L}(x)}{x-z}P_n^{2}(x)w_{_L}(x)dx,
\end{align}
where $v_{_L}(x)=-\ln w_{_L}(x)$ with $w_{_L}(x)$ given by \eqref{wL}.
\end{lemma}	
\begin{proof}
We only present the derivation of \eqref{lemma21} since \eqref{CPn2'} can be obtained via a similar argument.

From definition \eqref{defC}, it follows that		
\begin{align}
C(P_{n-1}P_{n}w_{_L})'=& \frac{d}{dz}\int_{0}^{+\infty}\frac {P_{n-1}(x)P_{n}(x)w_{_L}(x)}{x-z}dx\nonumber\\
=&\int_{0}^{+\infty}\frac{P_{n-1}(x)P_{n}(x)w_{_L}(x)}{(x-z)^{2}}dx\nonumber\\
=&-\int_{0}^{+\infty}P_{n-1}(x)P_{n}(x)w_{_L}(x)\,d\left(\frac{x}{z(x-z)}\right).\label{L-C-d}
\end{align}	
Through integration by parts, we get
\begin{align}		 C(P_{n-1}P_{n}w_{_L})'=&-\left.P_{n-1}(x)P_{n}(x)\frac{x w_{_L}(x)}{z(x-z)}\right|_{x=0}^{x=+\infty}
+\int_{0}^{+\infty}P'_{n-1}(x)P_{n}(x)w_{_L}(x)\frac{x dx}{z(x-z)}\nonumber\\			 &+\int_{0}^{+\infty}P_{n-1}(x)P'_{n}(x)w_{_L}(x)\frac{x dx}{z(x-z)}\nonumber\\
&-\int_{0}^{+\infty}\frac{x v_{_L}'(x)}{z(x-z)}P_{n-1}(x)P_{n}(x)w_{_L}(x)dx.\label{L-C-ibp}
\end{align}	
Now we analyze the first three terms on the r.h.s. one by one. Since $w_0(0)$ is bounded, as stated in Assumption \ref{ass:w0} from the Introduction, we have $\left.xw_{_L}(x)\right|_{x=0}=\left.x^{\lambda+1}w_0(x)\right|_{x=0}=0$ for $\lambda>-1$.  This fact together with Assumption \ref{ass:wL} indicates that the first term is zero. To continue, noting that
$\frac{x}{z(x-z)}=\frac{1}{x-z}+\frac{1}{z}$,
in view of \eqref{or-1}, we find that the second term reduces to $C(P'_{n-1}P_{n}w_{_L})$.
Similarly, on account of \eqref{or-2}, the third term becomes $C(P_{n-1}P'_{n}w_{_L})+\frac{n\cdot h_{n-1}}{z}$.
Integrating the above analysis, we  come to \eqref{lemma21}.
	\end{proof}	

{\sloppy \begin{remark}\label{L-R2}
From the above argument, we see that the construction of  the differential $d\big(\frac{x}{z(x-z)}\big)$  in  \eqref{L-C-d} is crucial for our derivation. It ensures the presence of $x$ in front of $w_{_{L}}(x)$ in the first term on the r.h.s. of \eqref{L-C-ibp}, which guarantees that this term is zero at $x=0$ for $\lambda>-1$.
When $\lambda>0$, the differential $d\left(\frac{1}{x-z}\right)$ was formed \cite{WvA18} instead of $d\left(\frac{x}{z(x-z)}\right)$ in \eqref{L-C-d}, namely,
\begin{align*}
C\left(P_{n-1}P_nw_{_L}\right)'=&\int_{0}^{+\infty}\frac{P_{n-1}(x)P_{n}(x)w_{_L}(x)}{(x-z)^{2}}dx
=-\int_{0}^{+\infty}P_{n-1}(x)P_{n}(x)w_{_L}(x)\,d\left(\frac{1}{x-z}\right)\\
=&-\left.P_{n-1}(x)P_{n}(x)\frac{ w_{_L}(x)}{x-z}\right|_{x=0}^{x=+\infty}+ C(P_{n-1}'P_nw_{_L})+C(P_{n-1}P_n'w_{_L})\\
&-\int_0^{+\infty}\frac{v'_{_L}(x)}{x-z}P_{n-1}(x)P_n(x)w_{_L}(x)dx.
\end{align*}
Note that $w_{_L}(x)$ evaluates to zero at $0$ when $\lambda>0$. The above expression and the one for $C(P_{n}^{2}w_{_L})'$ which can be obtained via a similar argument differ from \eqref{lemma21}-\eqref{CPn2'}, resulting in a different form of the four elements in $R(z)$ and consequently of $A_n$ and $B_n$. Refer to \cite[Section 4.3]{WvA18} for a detailed analysis.
\end{remark}}	

Now we proceed to compute $R(z)$ by using \eqref{rmatrixc} and the property \eqref{Cpn}.
\begin{lemma}\label{R-L}
 The quantities $R_{1,1}, R_{1,2}, R_{2,1}$ and $ R_{2,2}$ are given by
	\begin{align}			 R_{1,1}(z)=&-R_{2,2}(z)=-\frac{1}{z}\left(\frac{1}{h_{n-1}}\int_{0}^{+\infty}\frac{xv_{_L}'(x)}{x-z}P_{n-1}(x)P_n(x)w_{_L}(x)dx-n\right),\label{R11}\\			 R_{1,2}(z)=&-\frac{1}{2\pi i}\cdot\frac{1}{z}\int_{0}^{+\infty}\frac{xv_{_L}'(x)}{x-z}P_n^{2}(x)w_{_L}(x)dx,\label{R12}\\
R_{2,1}(z)=&\frac{2\pi i}{h_{n-1}^{2}}\cdot \frac{1}{z}\int_{0}^{+\infty}\frac{xv_{_L}'(x)}{x-z}P_{n-1}^{2}(x)w_{_L}(x)dx.\label{r21}		 \end{align}
	\end{lemma}
	
\begin{proof}
It is well-known the derivative of $\det Y(z)$ is given by \cite[(0.8.10.1)]{HJ}
\[\frac{d}{dz} \det Y(z)=\det Y(z)\cdot \tr\left(Y'(z)Y^{-1}(z)\right).\]
Since $\det Y(z)=1$ and $Y'(z)Y^{-1}(z)=R(z)$, it follows that $\tr(R(z))=0$, i.e. $R_{1,1}=-R_{2,2}$. From \eqref{rmatrixc}, $R_{1,1}(z)$ takes the form
\begin{align}\label{r11-1}			 R_{1,1}(z)=&\frac{1}{h_{n-1}}\left(-P_n'(z)C(P_{n-1}w_{_L})+P_{n-1}(z)C(P_nw_{_L})'\right).		 		 \end{align}	
According to \eqref{Cpn}, one has
\begin{align}	 P'_{n}(z)C(P_{n-1}w_{_{L}})=&C(P_{n-1}P'_{n}w_{_{L}}),\label{orthogonalc}	 \\
P'_{n-1}(z)C(P_{n}w_{_{L}})=&C(P'_{n-1}P_{n}w_{_{L}}),\label{DC-3}
\end{align}
	and
	\begin{equation*}	 P_{n-1}(z)C(P_{n}w_{_{L}})=C(P_{n-1}P_{n}w_{_{L}}).
	\end{equation*}
Differentiating both sides of the last w.r.t. $z$, in view of \eqref{DC-3}, we get
	\begin{equation*}	 P_{n-1}C(P_{n}w_{_{L}})'=C(P_{n-1}P_{n}w_{_{L}})'-C(P'_{n-1}P_{n}w_{_{L}}).
	\end{equation*}
Plugging it and \eqref{orthogonalc} into \eqref{r11-1} yields
\begin{align*}			 R_{1,1}=&\frac{1}{h_{n-1}}\left(C(P_{n-1}P_nw_{_L})'-C(P'_{n-1}P_nw_{_L})-C(P_{n-1}P_n'w_{_L})\right).
\end{align*}
Combining it with \eqref{lemma21} leads us to \eqref{R11}.
	
Expression \eqref{R12} can be derived similarly. Specifically, by \eqref{rmatrixc}, we have
\begin{align}\label{R12c}
R_{1,2}=&\frac{1}{2\pi i}\left(-P'_nC(P_nw_{_L})+P_nC(P_nw_{_L})'\right).
\end{align}
From \eqref{Cpn}, it follows that
\begin{align}
P_n'C(P_nw_{_L})=&C(P_n'P_nw_{_L}),\label{eq4}\\
P_nC(P_nw_{_L})=&C(P_{n}^{2}w_{_L}).\nonumber
\end{align}
We differentiate both sides of the latter w.r.t. $z$ and find, in light of \eqref{eq4},
	 \begin{equation}\label{eq6}	 	 P_nC(P_nw_{_L})'=C(P_{n}^{2}w_{_L})'-C(P'_nP_nw_{_L}).
	 \end{equation}
Inserting \eqref{eq4}-\eqref{eq6} into \eqref{R12c}, and substituting \eqref{CPn2'} into the resulting expression, we arrive at \eqref{R12}.

To continue, \eqref{rmatrixc} gives us
	\begin{equation*}
    R_{2,1}=\frac{-2\pi i}{h_{n-1}^{2}}\left(-P'_{n-1}C(P_{n-1}w_{_L})+P_{n-1}C(P_{n-1}w_{_L})'\right).		 \end{equation*}
Comparing it with \eqref{R12c}, we observe that $\frac{h_{n-1}^2}{-2\pi i}R_{2,1}$ equals $2\pi i R_{1,2}$ where $n$ should be replaced by $n-1$. Hence, \eqref{r21} follows from \eqref{R12}.
\end{proof}
	
We are now in a position to construct the lowering and raising operators for $P_n(z)$, given by \eqref{lo}-\eqref{ro}. By substituting \eqref{R11}-\eqref{R12} into \eqref{p'nyr},  we find
\begin{align}			
P'_n(z)=&R_{1,1}P_n(z)-\frac{2\pi i}{h_{n-1}}R_{1,2}P_{n-1}(z)\nonumber\\
=&-\frac{1}{z}\left(\frac{1}{h_{n-1}}\int_{0}^{+\infty}\frac{zv_{_L}'(z)-xv_{_L}'(x)}{z-x}P_{n-1}(x)P_n(x)w_{_L}(x)dx
			-n\right)\cdot P_n(z)\nonumber\\			 &+\frac{1}{z}\cdot\frac{1}{h_{n-1}}\int_{0}^{+\infty}\frac{zv_{_L}'(z)-xv_{_L}'(x)}{z-x}P_n^{2}(x)w_{_L}(x)dx\cdot P_{n-1}(z)\nonumber\\
&+\frac{v_{_L}'(z)}{h_{n-1}}\left[-P_n(z)C(P_{n-1}P_nw_{_L})+P_{n-1}(z)C(P_n^2w_{_L})\right].\label{Pn'1}
\end{align}	
According to \eqref{Cpn}, we have
\[C(P_{n-1}P_nw_{_L})=P_{n-1}C(P_nw_{_L}),\quad\qquad C(P_n^2w_{_L})=P_nC(P_nw_{_L}). \]
Combining them indicates that the term in the square bracket in \eqref{Pn'1} is zero. Noting that $\beta_n=h_n/h_{n-1}$, we finally come to the lowering operator \eqref{lo} with $A_n$ and $B_n$ given by \eqref{AN}-\eqref{BN}.
	
The raising operator can be derived following a similar reasoning. Indeed, by plugging the expressions of $R_{2,1}$ and $R_{2,2}$ given by \eqref{r21} and \eqref{R11} into \eqref{pn-1'y}, in view of the identity $C(P_{n-1}^2w_{_L})=P_{n-1}C(P_{n-1}w_{_L})$
which results from \eqref{Cpn}, and on account of \eqref{I1}, we arrive at \eqref{ro}.

\begin{remark}\label{lr-r}
The compatibility conditions $(S_1)$ and $(S_2)$ can be derived by using the ladder operators \eqref{lo}-\eqref{ro} and the recurrence relation \eqref{3r}.
A combination of $(S_1)$ and $(S_2)$ yields $(S_2')$. See \cite[Theorem 4.2]{WvA18} for a detailed derivation.
Although the expressions of $A_n$ and $B_n$ therein (c.f. \cite[(4.4)-(4.5)]{WvA18}), valid for Laguerre-type orthogonal polynomials with $\lambda>0$, differ from ours given by \eqref{AN}-\eqref{BN}, the proof therein still holds for our problem since it does not depend on the specific forms of $A_n$ and $B_n$.
\end{remark}

\subsection{Examples and extensions}\label{exL}
We first consider the monic Laguerre polynomials to demonstrate the validity of Theorem \ref{main}.
In this subsection, we assume $x\in[0,+\infty)$, unless otherwise specified.
		
\begin{exm}
For the classical Laguerre weight function, i.e. $w_{_L}(x)=x^{\lambda}\e^{-x}, \lambda>-1$, we have
\begin{equation*}
\frac{zv'_{_L}(z)-xv'_{_L}(x)}{z-x}=1.
\end{equation*}
Substituting it into \eqref{AN}-\eqref{BN} gives us
\begin{align*}
A_n(z)=\frac{1}{z},\qquad\qquad	
B_n(z)=-\frac{n}{z}.
\end{align*}
Inserting them into $(S_1)$ and $(S^{'}_2)$, we get
\begin{align*}
\alpha_n=2n+\lambda+1,\qquad\qquad
\beta_n=n(n+\lambda),
\end{align*}
so that the three-term recurrence relation \eqref{3r} reads
\begin{equation}\label{r3-1}
P_{n+1}(x)=\left(x-(2n+\lambda+1)\right)P_n(x)-n(n+\lambda)P_{n-1}(x),\qquad n=0,1,\cdots,
\end{equation}
with $P_0(x):=1$ and $P_{-1}(x):=0$. Using it, we write down the monic Laguerre polynomials $\{P_1,P_2,P_3,P_4\}$ as below:
\begin{align*}
P_1(x)=&x-(\lambda+1),\nonumber\\
				 P_2(x)=&x^2-2(\lambda+2)x+(\lambda+1)(\lambda+2),\nonumber\\
				 P_3(x)=&x^3-3(\lambda+3)x^2+3(\lambda+2)(\lambda+3)x-(\lambda+1)(\lambda+2)(\lambda +3),\nonumber\\
				 P_4(x)=&x^4-4(\lambda+4)x^3+6(\lambda+3)(\lambda+4)x^2-4(\lambda+2)(\lambda+3)(\lambda+4)x\\
&+(\lambda+1)(\lambda+2)(\lambda +3)(\lambda+4).			 
\end{align*}
Noting that $\beta_j=h_j/h_{j-1}$ for $j=1,\cdots,n$ and $h_0:=\int_{0}^{+\infty}x^\lambda \e^{-x}dx=\Gamma(\lambda+1)$, we find
\begin{equation}\label{hn-1}
h_n=h_0\beta_1\beta_2\cdots\beta_n=n! \cdot \Gamma(n+\lambda+1).
\end{equation}
Since $D_n=\prod_{j=0}^{n-1}h_j$, it follows that
\begin{align*}
D_n:=&\det \left( \int_0^{+\infty}x^{i+j}x^{\lambda}\e^{-x}dx\right) _{i,j=0}^{n-1}\\
=&\prod_{j=0}^{n-1}j!\cdot \Gamma(j+\lambda+1)=\frac{G(n+1)\cdot G(n+\lambda+1)}{G(\lambda+1)},
\end{align*}
where the Barnes $G$-function $G(z)$ is defined by \cite[p.184]{Forrester}
\begin{equation*}
G(z+1)=\Gamma(z)G(z),\qquad\qquad G(1):=1.
\end{equation*}
Since the classical Laguerre polynomials $L_n^{\lambda}(x)=(-1)^n/n!\cdot P_n(x)$ (c.f. \cite[(4.6.3)]{Ismail}), we observe that \eqref{r3-1} and \eqref{hn-1} coincide with (4.6.26) and (4.6.2) of \cite{Ismail} respectively. See also (5.1.10) and (5.1.1) of \cite{Szego}.
\end{exm}

We now examine four classes of semi-classical Laguerre-type weights and one singularly perturbed Laguerre weight, and generalize the existing results from $\lambda>0$ to $\lambda>-1$.  These results were originally derived by substituting $A_n$ and $B_n$ into $(S_1), (S_2)$ and $(S_2')$, and combining the obtained equations with the ones deduced by differentiating the orthogonality relations
 \begin{gather}	 \int_{0}^{+\infty}P_n^2(x)w_{_L}(x)dx=h_n,\label{or-Pn2-L}\\
 \int_{0}^{+\infty}P_n(x)P_{n-1}(x)w_{_L}(x)dx=0,\label{Pnn-1-L}
\end{gather}
w.r.t. the new variable introduced in $w_{_L}(x)$, not w.r.t.  $x$. When $\lambda>-1$, we make use of \eqref{AN}-\eqref{BN} to compute $A_n$ and $B_n$, and differentiate \eqref{or-Pn2-L}-\eqref{Pnn-1-L}. We verify that the resulting expressions and identities coincide with those for the $\lambda>0$ case. Consequently, the prior results for $\lambda>0$ remain valid for $\lambda>-1$.

\begin{exm}\label{ex-2-L}
The Hankel determinant generated by the weight function $w_{_L}(x)=x^{\lambda}\e^{-x}(x+t)^{\gamma}$, $\lambda,t>0,\gamma\in\mathbb{R}$ arises from the single-user multiple-input multiple-output wireless communication system and was studied in \cite{ChenMcKay12}. We assume $\lambda>-1$ and note that
\begin{equation*}
\frac{zv'_{_L}(z)-xv'_{_L}(x)}{z-x}=1-\frac{\gamma t}{(x+t)(z+t)}.
\end{equation*}
Substituting it into \eqref{AN}-\eqref{BN}, we get
\begin{align*}
A_n(z)=\frac{1-R_n(t)}{z}+\frac{R_n(t)}{z+t}, \qquad\qquad B_n(z)=-\frac{n+r_n(t)}{z}+\frac{r_n(t)}{z+t},
\end{align*}
where
\[R_n(t):=\frac{\gamma}{h_n}\int_0^{+\infty}\frac{P_n^2(x;t)w_{_L}(x)}{x+t}dx,\quad\qquad r_n(t):=\frac{\gamma}{h_{n-1}}\int_0^{+\infty}\frac{P_n(x;t)P_{n-1}(x;t)w_{_L}(x)}{x+t}dx.\]
\sloppy{The expressions of $A_n$ and $B_n$ coincide with those in \cite{ChenMcKay12} where $\lambda>0$. Moreover, we differentiate the orthogonality relations
\begin{gather*}
h_n(t)=\int_0^{+\infty}P_n^2(x;t)x^{\lambda}\e^{-x}(x+t)^{\gamma}dx,\\
0=\int_0^{+\infty}P_n(x;t)P_{n-1}(x;t)x^{\lambda}\e^{-x}(x+t)^{\gamma}dx,
\end{gather*}
w.r.t. $t$ and obtain the same identities as in \cite{ChenMcKay12} (c.f. (238) and (241)). Hence, we conclude that the Jimbo-Miwa-Okamoto $\sigma$-form of Painlev\'{e} V equation satisfied by the logarithmic derivative of the Hankel determinant with $\lambda>0$ (see \cite[(69)]{ChenMcKay12}) is also valid for $\lambda>-1$.} 

The Hankel determinant for the weight $w_{_{L}}(x)=x^{\lambda}\e^{-x}\left(\frac{x+t}{x+T}\right)^{N_s},\lambda,t,T,N_s>0$
characterizes the signal-to-noise ratio which determines the performance of the multiple-antenna wireless communication system and was investigated in \cite{ChenHaqMcKay13}. Via an argument similar to the above, we confirm that the PDE (c.f. (46) therein) for the logarithmic derivative of the Hankel determinant, which can be treated as a two-variable generalization of the $\sigma$-form of Painlev\'{e} V equation, remains valid for $\lambda>-1$.

The weight function $w_{_L}(x)=x^{\lambda}\e^{-x}\prod\limits_{k=1}^N(x+t_k)^{\lambda_k},\lambda,t_k>0, \lambda_k\in\mathbb{R}$ was considered in \cite{MuLyu2024}, and the two weights mentioned above correspond to $N=1,2$. Assuming $\lambda>-1$ and plugging
\[\frac{zv'_{_L}(z)-xv'_{_L}(x)}{z-x}=1-\sum_{k=1}^N\frac{\lambda_kt_k}{(z+t_k)(x+t_k)}\] into \eqref{AN}-\eqref{BN} yields the expressions of $A_n$ and $B_n$, which are identical to those given by (2.6)-(2.7) in \cite{MuLyu2024}. It is easy to check that the differential relations (3.1)-(3.2) in \cite{MuLyu2024} remain valid for $\lambda>-1$. Consequently, the second order PDE deduced in \cite{MuLyu2024} for the logarithmic derivative of the associated Hankel determinant with $\lambda>0$ continues to hold when $\lambda>-1$.

The recurrence coefficients of the monic orthogonal polynomials and the Hankel determinant associated with $w_{_L}(x)=x^{\lambda}\e^{-N(x+s(x^2-x))}$, $\lambda,N>0,s\in[0,1]$ were analyzed in \cite{HanChen17}. Supposing $\lambda>-1$, we can verify that the expressions of $A_n$ and $B_n$ obtained by using \eqref{AN}-\eqref{BN} and the differential identities deduced by taking the derivative of the orthogonality relations w.r.t. $s$ coincide with those in \cite{HanChen17} (c.f. (2.6)-(2.7), (2.15) and (2.18)). Therefore, the second order differential equation satisfied by the recurrence coefficient $\alpha_n(s)$ and the expression for the logarithmic derivative of the Hankel determinant in terms of $\alpha_n(s)$ and $\alpha_n'(s)$, which were established in \cite{HanChen17} for $\lambda>0$ (see (2.23), (2.27) and (4.5)), can be generalized to $\lambda>-1$.
\end{exm}

\begin{exm}
The Hankel determinant for  $w_{_{L}}(x)=x^{\lambda}\e^{-x-s/x},\lambda,s>0$ was represented by an integral involving the Painlev\'{e} III transcendent in \cite{ChenIts2010}. This conclusion remains valid for $\lambda>-1$, since the expressions of $A_n$ and $B_n$ obtained via \eqref{AN}-\eqref{BN} and the differential equalities derived by differentiating the orthogonality relations for $\lambda>-1$ are consistent with those for $\lambda>0$ (c.f. (2.7)-(2.8), (3.1) and (3.5) in \cite{ChenIts2010}).
\end{exm}

\begin{remark}
While verifying the above examples, we find that the calculations of $A_n$ and $B_n$ using \eqref{AN}-\eqref{BN} are more straightforward than those using \eqref{a-n}-\eqref{b-n}.
\end{remark}

By adapting the argument in the preceding subsection, we construct the ladder operators for the monic Laguerre-type orthogonal polynomials containing several jump discontinuities or a Fisher-Hartwig singularity with simultaneous jump and root-type behavior.

\begin{theorem}\label{L-jump}
Denote the Laguerre-type weight function with jump discontinuities by
\begin{align}\label{wh}
\widehat{w}(x;\vec{t}\,):= w_{_L}(x)\left(\omega_0 + \sum\limits_{k=1}^{m} \omega_k \theta (x-t_k)\right),
\end{align}
where $w_{_L}(x)$ is defined by \eqref{wL} and independent of $\vec{t}=(t_1,\cdots, t_m)$ with $0\leq t_1<\cdots<t_m<\infty$,  $\sum\limits_{k=0}^{\ell}\omega_k\geq0$ for $0\leq\ell\leq m$, and $\theta(x)$ is the Heaviside step function which is $1$ for $x>0$ and $0$ otherwise. The associated monic orthogonal polynomials satisfy the ladder operators \eqref{lo}-\eqref{ro} with $A_n$ and $B_n$ given by
\begin{subequations}\label{AnBn-2}
\begin{gather}\label{An-2}			 A_n(z)=\frac{1}{z}\left[\frac{1}{h_n}\int_{0}^{\infty}\frac{zv'_{_L}(z)-xv'_{_L}(x)}{z-x}P_n^{2}(x)\widehat{w}(x;\vec{t}\,)dx-\sum_{k=1}^{m}R_{n,k}\right]+\sum_{k=1}^{m}\frac{R_{n,k}}{z-t_k},	 	 \end{gather}
\begin{equation}\label{Bn-2}
\begin{aligned}
B_n(z)=&\frac{1}{z}\left[\frac{1}{h_{n-1}}\int_{0}^{\infty}\frac{zv'_{_L}(z)-xv'_{_L}(x)}{z-x}P_{n-1}(x)P_n(x)\widehat{w}(x;\vec{t}\,)dx
 -\sum_{k=1}^{m}r_{n,k}-n\right]\\
 &+\sum_{k=1}^{m}\frac{r_{n,k}}{z-t_k},
\end{aligned}
\end{equation}
\end{subequations}
where $v_{_L}(x):=-\ln w_{_L}(x)$ and $\{R_{n,k}, r_{n,k}, k=1,\cdots,m\}$ read
\begin{align*}
		R_{n,k}(\vec{t}\,):=\frac{w_k w_{_L}(t_k)P_n^2(t_k)}{h_n},\qquad\qquad r_{n,k}(\vec{t}\,):=\frac{w_kw_{_L}(t_k)P_n(t_k)P_{n-1}(t_k)}{h_{n-1}}.
\end{align*}	
\end{theorem}

\begin{exm}	
The weight function $\widehat{w}(x;\vec{t}\,)$ with $w_{_L}(x)=x^{\lambda}\e^{-x},\lambda>0$ was explored in \cite{LyuChenXu2023} and the expressions \eqref{AnBn-2} were given in Remark 10 therein. When $m=2$ with $\omega_0=0,\omega_1=1,$ $\omega_2=-1$, i.e.
\[\widehat{w}(x;t_1,t_2)= x^{\lambda}\e^{-x}\left(\theta (x-t_1)-\theta (x-t_2)\right),\]
the corresponding Hankel determinant represents the probability that all the eigenvalues of the Laguerre unitary ensemble lie in the interval $(t_1,t_2)$, up to a constant, and was investigated in \cite{BasorChenZhang12} where $\lambda>0$ and $\{a,b\}$ were used instead of $\{t_1,t_2\}$. For this case, according to \eqref{AnBn-2}, we have
\begin{align*}		 A_n(z)=&\frac{R_n}{z}+\frac{R_{n,1}}{z-t_1}+\frac{R_{n,2}}{z-t_2},\\		 	 B_n(z)=&\frac{r_n}{z}+\frac{r_{n,1}}{z-t_1}+\frac{r_{n,2}}{z-t_2},
\end{align*}
where $R_n:=1-R_{n,1}-R_{n,2}$ and $r_n:=-n-r_{n,1}-r_{n,2}$. When $\lambda>0$, via integration by parts, we have
\[R_n=\frac{\lambda}{h_n}\int_{t_1}^{t_2}P_n^2(x)x^{\lambda-1}\e^{-x}dx,\qquad\qquad r_n=\frac{\lambda}{h_{n-1}}\int_{t_1}^{t_2}P_n(x)P_{n-1}(x)x^{\lambda-1}\e^{-x}dx.\] Hence, our expressions of $\{A_n,B_n\}$ are consistent with (4.1)-(4.2) in \cite{BasorChenZhang12} where $\lambda>0$. In addition, the differential relations given by (4.27) in \cite{BasorChenZhang12} remain valid for $\lambda>-1$. Therefore, the two-variable generalization of Painlev\'{e} V system established in \cite{BasorChenZhang12} for the logarithmic derivative of the probability can be extended to $\lambda>-1$.

When $w_{_L}(x)=x^{\lambda}\e^{-x-t/x}$ and $m=1$, \eqref{wh} becomes \[\widehat{w}(x;t_1)=x^{\lambda}\e^{-x-t/x}\left(\omega_0+\omega_1\theta(x-t_1)\right),\] and
\[\frac{zv'_{_L}(z)-xv'_{_L}(x)}{z-x}=1+\frac{t}{zx}.\]
 According to \eqref{AnBn-2}, we get
\begin{align*}
A_n(z)=&\frac{tR_n}{z^2}+\frac{1-R_{n,1}}{z}+\frac{R_{n,1}}{z-t_1},\qquad \\ B_n(z)=&\frac{t r_n}{z^2}-\frac{r_{n,1}+n}{z}+\frac{r_{n,1}}{z-t_1},
\end{align*}
where
\[R_n:=\frac{1}{h_n}\int_0^{+\infty}P_n^2(x)\widehat{w}(x)\frac{dx}{x},\qquad\qquad r_n:=\frac{1}{h_{n-1}}\int_0^{+\infty}P_n(x)P_{n-1}(x)\widehat{w}(x)\frac{dx}{x}.\]
When $\omega_0=0$ and $\omega_1=1$, the above expressions match \cite[(2.1)]{LyuGriffinChen2019} where $\lambda>0$, and the differential relations (2.19) and (2.21) therein remain valid for $\lambda>-1$. Hence, the second order sixth degree PDE satisfied by the logarithmic derivative of the associated Hankel determinant (c.f. (2.40) in \cite{LyuGriffinChen2019}) still holds for $\lambda>-1$.
\end{exm}

\begin{theorem}
The monic orthogonal polynomials associated with the weight function
\begin{align*}
w(x)=w_{_{L}}(x)|x-t|^{\gamma}(A+B\theta(x-t)),\qquad x\in[0,+\infty),~t\geq0,~\gamma>0,
\end{align*}
with $w_{_L}(x)$ defined by \eqref{wL}, satisfy the ladder operators \eqref{lo}-\eqref{ro} with $A_n$ and $B_n$ reading
\begin{subequations}\label{AnBn-L5-1}
\begin{equation}
\begin{aligned}
A_n(z)=&\frac{1}{z}\left[\frac{1}{h_n}\int_0^{+\infty}\frac{zv_{_L}'(z)-xv_{_L}'(x)}{z-x}P_n^2(x)w(x)dx
-\frac{\gamma}{h_n}\int_0^{+\infty}\frac{P_n^2(x) }{x-t}w(x)dx\right]\\
&+\frac{\gamma}{ h_n}\int_0^{+\infty}\frac{P_n^2(x)w(x)}{(z-x)(x-t)}dx,
\end{aligned}
\end{equation}
\begin{equation}
\begin{aligned}
B_n(z)=&\frac{1}{z}\left[\frac{1}{h_{n-1}}\int_0^{+\infty}\frac{zv_{_L}'(z)-xv_{_L}'(x)}{z-x}P_n(x)P_{n-1}(x)w(x)dx\right.\\
&\left.\quad-\frac{\gamma}{h_{n-1}} \int_0^{+\infty}\frac{P_n(x)P_{n-1}(x)}{x-t}w(x)dx-n\right]+\frac{\gamma}{h_{n-1}}\int_0^{+\infty}\frac{P_n(x)P_{n-1}(x)}{(z-x)(x-t)}w(x)dx,
\end{aligned}
\end{equation}
\end{subequations}
where $v_{_L}(x):=-\ln w_{_L}(x)$. When $w_{_L}(x)=x^{\lambda}\e^{-x}$ with $\lambda>-1$, we have
\begin{subequations}\label{AnBn-L5-2}
\begin{align}
A_n(z)=&\frac{1-R_n(t)}{z}+\frac{\gamma}{ h_n}\int_0^{+\infty}\frac{P_n^2(x)w(x)dx}{(z-x)(x-t)},\\
B_n(z)=&-\frac{r_n(t)+n}{z}+\frac{\gamma}{h_{n-1}}\int_0^{+\infty}\frac{P_n(x)P_{n-1}(x)}{(z-x)(x-t)}w(x)dx,
\end{align}
\end{subequations}
where
\[R_n(t):=\frac{\gamma}{h_n}\int_0^{+\infty}\frac{P_n^2(x)}{x-t}w(x)dx,\qquad\qquad r_n(t):=\frac{\gamma}{h_{n-1}} \int_0^{+\infty}\frac{P_n(x)P_{n-1}(x)}{x-t}w(x)dx.\]
\end{theorem}

\begin{exm}\label{L-FH}
The ladder operators for the monic orthogonal polynomials associated with $\widetilde{w}(x)|x-t|^{\gamma}(A+B\theta(x-t)),~ x,t\in[a,b],~\gamma>0$,
with $\widetilde{w}(x)$ being a smooth function and $\widetilde{w}(a)=\widetilde{w}(b)=0$, was established in \cite{MinChenRMTA}.
In particular, the Hankel determinant generated by the moments of this weight function with $\widetilde{w}(x)=x^{\lambda}\e^{-x}, x\in[0,+\infty),\lambda>0$, was analyzed therein.

When $\lambda>0$, through integration by parts, we find that the weight function $w(x)$ defined in the above theorem satisfies the following condition
\[\int_0^{+\infty}P_n^2(x)w'(x)dx=0.\]
Inserting
\begin{align*}
w'(x)=&-v_{_L}'(x)w(x)+w_{_L}(x)|x-t|^{\gamma}B\delta(x-t)\\
&+w_{_L}(x)\left(\delta(x-t)((x-t)^{\gamma}-(t-x)^{\gamma})+\gamma\frac{|x-t|^{\gamma}}{x-t}\right)(A+B\theta(x-t))
\end{align*}
into the above integral gives us
\[\int_0^{+\infty}P_n^2(x)v_{_L}'(x)w(x)dx=\gamma\int_0^{+\infty}\frac{P_n^2(x) }{x-t}w(x)dx.\]
Using it,  we observe that $A_n$ and $B_n$ given by \eqref{AnBn-L5-1} coincide with those in Theorem 2.1 of \cite{MinChenRMTA}. In particular, through integration by parts, we note that \eqref{AnBn-L5-2} agree with (3.1)-(3.2) in their work. Moreover, the expressions of $A_n(z)$ and $B_n(z)$ for $z\rightarrow\infty$ (c.f. (3.3)-(3.4) therein) are still valid.  In addition, the differential relations (4.1) and (4.5) in \cite{MinChenRMTA} continue to hold for $\lambda>-1$. Therefore, the results established in \cite{MinChenRMTA} for the Hankel determinant $D_n$ generated by $x^{\lambda}\e^{-x}|x-t|^{\gamma}(A+B\theta(x-t))$ with $\lambda>0$ can be generalized to $\lambda>-1$, including the integral representation for $D_n$ in terms of Painlev\'{e} V transcendent, and the second order difference equation and the $\sigma$-form of Painlev\'{e} V satisfied by the logarithmic derivative of $D_n$. See Theorems 4.3 and 4.4 in \cite{MinChenRMTA}.
\end{exm}

\section{Ladder Operators for Jacobi-type Orthogonal\\ Polynomials with $\alpha,\beta>-1$}
We turn our attention to the monic Jacobi-type orthogonal polynomials associated with \eqref{wJ}, namely,
\begin{align*}
w_{_{J}}(x)=&(1-x)^{\alpha}(1+x)^{\beta}w_0(x),\qquad x\in[-1,1],\quad\alpha,\beta>-1,
\end{align*}
where, as assumed in Introduction, $w_0(x)$ is continuously differentiable and bounded at $\pm1$.
By adapting the derivation technique of the preceding section to this case, we establish the ladder operators.

\begin{theorem}\label{main-J}
The monic Jacobi-type orthogonal polynomials associated with $w_{_J}(x)$ satisfy the ladder operators:
\begin{align}
P_n'(z)=&-B_n(z)P_n(z)+\beta_nA_n(z)P_{n-1}(z),\label{loJ}\\
P_{n-1}'(z)=&\left(B_n(z)+v_{_J}'(z)\right)P_{n-1}(z)-A_{n-1}(z)P_n(z),\label{roJ}
\end{align}
for $n\geq0$, where $v_{_J}(x):=-\ln w_{_J}(x)$. The quantities $A_n(z)$ and $B_n(z)$ are given as follows:
\begin{align}\label{An-J}
A_n(z):=\frac{1}{1-z^2}\left[  \frac{1}{h_n}\int_{-1}^{1}\frac{(1-z^2)v_{_J}'(z)-(1-x^2)v_{_J}'(x)}{z-x}P_n^2(x)w_{_J}(x)dx+2n+1 \right],
\end{align}
for $n\geq0$ with $A_{-1}(z):=0$, and
\begin{align}\label{Bn-J}		 B_n(z):=&\frac{1}{1-z^2}\cdot\frac{1}{h_{n-1}}\int_{-1}^{1}\frac{(1-z^2)v_{_J}'(z)-(1-x^2)v_{_J}'(x)}{z-x}P_n(x)P_{n-1}(x)w_{_J}(x)dx\nonumber\\
&+\frac{nz-\mathbf{p}(n)}{1-z^2}, \end{align}
for $n\geq1$ with $B_0(z):=0$. Furthermore, $A_n(z)$ and $B_n(z)$ satisfy $(S_1), (S_2)$ and $(S_2')$.
\end{theorem}

\begin{remark}\label{re-J-1}
When $\alpha,\beta>0$, \eqref{An-J}-\eqref{Bn-J} coincide with \eqref{a-n}-\eqref{b-n}. Actually, \eqref{Bn-J} is equivalent to
\begin{align}
 B_n(z)=&\frac{1}{h_{n-1}}\int_{-1}^{1}\frac{v_{_J}'(z)-\frac{1-x^2}{1-z^2}v_{_J}'(x)}{z-x}P_n(x)P_{n-1}(x)w_{_J}(x)dx+\frac{nz-\mathbf{p}(n)}{1-z^2}\nonumber\\
 =&\frac{1}{h_{n-1}}\int_{-1}^{1}\frac{v_{_J}'(z)-v_{_J}'(x)}{z-x}P_n(x)P_{n-1}(x)w_{_J}(x)dx\nonumber\\
 &-\frac{1}{1-z^2}\cdot\frac{1}{h_{n-1}}\int_{-1}^{1}(x+z)P_n(x)P_{n-1}(x)w_{_J}(x)v_{_J}'(x)dx+\frac{nz-\mathbf{p}(n)}{1-z^2}.	 \label{Bn-J-1}	
 \end{align}
When $\alpha,\beta>0$, we have
\begin{align}
&-\int_{-1}^{1}(x+z)P_n(x)P_{n-1}(x)w_{_J}(x)v_{_J}'(x)dx=\int_{-1}^{1}(x+z)P_n(x)P_{n-1}(x)dw_{_J}(x)\nonumber\\		 &=\left.(x+z)P_n(x)P_{n-1}(x)(1-x)^{\alpha}(1+x)^{\beta}w_0(x)\right|_{x=-1}^{x=1}-\int_{-1}^{1}P_n(x)P_{n-1}(x)w_{_J}(x)dx\nonumber\\
&\quad-\int_{-1}^{1}(x+z)P_n'(x)P_{n-1}(x)w_{_J}(x)dx
-\int_{-1}^{1}P_n(x)\left((x+z) P_{n-1}'(x)\right)w_{_J}(x)dx\nonumber\\
&=-\int_{-1}^{1}\left((x+z)P_n'(x)\right)P_{n-1}(x)w_{_J}(x)dx,\label{R5-1}
\end{align}
where the last identity is owing to \eqref{dp}-\eqref{or-1}.
Since $P_n(x)=x^n+\mathbf{p}(n)x^{n-1}+\cdots$ for $n\geq2$, we have
\vspace{-2mm}
\begin{align*}
P_n'(x)=&nx^{n-1}+(n-1)\mathbf{p}(n)x^{n-2}+\cdots\\
=&nP_{n-1}(x)+\{\text{linear combinations of } P_k(x),~0\leq k\leq n-2\},\\
xP_n'(x)=&nx^{n}+(n-1)\mathbf{p}(n)x^{n-1}+\cdots\\
=&nP_{n}(x)-\mathbf{p}(n)P_{n-1}(x)+\{\text{linear combinations of } P_k(x),~0\leq k\leq n-2\}.
\end{align*}
Hence, according to the orthogonality relation \eqref{dp}, we find
\begin{align}\label{R5-2}
\int_{-1}^{1}\left((x+z)P_n'(x)\right)P_{n-1}(x)w_{_J}(x)dx=&\left(nz-\mathbf{p}(n)\right)h_{n-1},
\end{align}
for $n\geq2$. When $n=1$, noting that $P_1'(x)=1$ and
\[-\mathbf{p}(1)=\alpha_0=\frac{1}{h_0}\int_{-1}^1xw_{_J}(x)dx,\]
where the first identity is due to \eqref{alp} and the second one results from \eqref{3r} with $j=0$. Hence, \eqref{R5-2} also holds for $n=1$.
Combining \eqref{R5-2} with \eqref{R5-1} and substituting the resulting equation into
\eqref{Bn-J-1}, we are led to \eqref{b-n}. By an analogous argument, we can show that, when $\alpha,\beta>0$, \eqref{An-J} is consistent with \eqref{a-n}.	
\end{remark}

\subsection{Proof of the ladder operators via RHP}\label{proof-J-RHP}
The following lemma is essential for the derivation.
\begin{lemma}\label{lemma-J}
$C(P_{n-1}P_nw_{_J})'$ and $C(P_n^2w_{_J})'$ can be expressed as
 \begin{align}
 &C(P_{n-1}P_nw_{_J})'=C(P_{n-1}'P_nw_{_J})+C(P_{n-1}P_n'w_{_J})\nonumber\\
 &\quad\qquad+\frac{1}{z^2-1}\left[\int_{-1}^{1}\frac{(1-x^2)v_{_J}'(x)}{x-z}P_{n-1}(x)P_n(x)w_{_J}(x)dx+\left(nz-\mathbf{p}(n)\right)h_{n-1}\right],\label{C(P_{n-1}P_nw)'}\\	 &C(P_n^2w_{_J})'= 2C(P_n'P_nw_{_J})+\frac{1}{z^2-1}\left[\int_{-1}^{1}\frac{(1-x^2)v_{_J}'(x)}{x-z}P_n^2(x)w_{_J}(x)dx+(2n+1) h_n\right].\label{lemma2-1}
\end{align}
\end{lemma}
\begin{proof}
From the definition of $C(P_{n-1}P_nw_{_J})$, it follows that     	 
\begin{align}		 C(P_{n-1}P_nw_{_J})'=&\dfrac{d}{dz}\int_{-1}^{1}\frac{P_{n-1}(x)P_n(x)w_{_J}(x)}{x-z}dx=	 \int_{-1}^{1}\frac{P_{n-1}(x)P_n(x)w_{_J}(x)}{(x-z)^2}dx\nonumber\\		 =&\int_{-1}^{1}{P_{n-1}(x)P_n(x)w_{_J}(x)}\left(\frac{1}{(x-z)^2}+\frac{1}{1-z^2}\right)dx\nonumber\\		 =&-\int_{-1}^{1}{P_{n-1}(x)P_n(x)w_{_J}(x)}\,d\left(\frac{1}{x-z}+\frac{x+z}{z^2-1}\right)\nonumber\\		 =&-\int_{-1}^{1}{P_{n-1}(x)P_n(x)w_{_J}(x)}\,d\left(\frac{x^2-1}{(x-z)(z^2-1)}\right),\label{J-C-d}
\end{align}
where the equality on the second line of the above equation is due to \eqref{dp}. To continue, we integrate by parts and get
    \begin{align}		 C(P_{n-1}P_nw_{_J})'=&\left.P_{n-1}(x)P_n(x)\frac{(1-x^2)w_{_J}(x)}{(x-z)(z^2-1)}\right|_{x=-1}^{x=1}\nonumber\\&+\int_{-1}^{1}P_{n-1}'(x)P_n(x)w_{_J}(x)\left(\frac{1}{x-z}+\frac{x+z}{z^2-1}\right)	 dx\nonumber\\
    &+\int_{-1}^{1}P_{n-1}(x)P_n'(x)w_{_J}(x)\left(\frac{1}{x-z}+\frac{x+z}{z^2-1}\right)	 dx\nonumber\\
    &+\frac{1}{z^2-1}\int_{-1}^{1}\frac{(1-x^2)v_{_J}'(x)}{x-z}P_{n-1}(x)P_n(x)w_{_J}(x)dx.\label{3.6}
\end{align}
For $\alpha,\beta>-1$, we have $(1-x^2)w_{_J}(x)|_{x=\pm1}=(1-x)^{\alpha+1}(1+x)^{\beta+1}w_0(x)|_{x=\pm1}=0$, since $w_0(x)$ is bounded at $\pm1$, as per Assumption \ref{ass:w0} from the Introduction. Hence, the first term on the r.h.s. of \eqref{3.6} is zero. In view of \eqref{or-1}, the second term is simplified to $C(P_{n-1}'P_nw_{_J})$.
On account of \eqref{R5-2}, the third term becomes $C(P_{n-1}P_n'w_{_J})+\frac{h_{n-1 }}{z^2-1}\left(nz-\textbf{p}(n)\right)$.
Combining the above analysis with \eqref{3.6}, we come to \eqref{C(P_{n-1}P_nw)'}.
The proof of \eqref{lemma2-1} is quite similar and so is left to the reader.
\end{proof}	

{\sloppy \begin{remark}
From the above examination, we note that the construction of  the differential $d\big(\frac{x^2-1}{(x-z)(z^2-1)}\big)$  in  \eqref{J-C-d} is critical for our derivation. It ensures that $w_{_{J}}(x)$ is multiplied by $(1-x^2)$ in the first term on the r.h.s. of \eqref{3.6}, which guarantees that this term is zero at $x=\pm1$ for $\alpha,\beta>-1$.
When $\alpha,\beta>0$, the differential $d\left(\frac{1}{x-z}\right)$ was instead introduced in \cite{WvA18}. See Remark \ref{L-R2} for the reasoning.
\end{remark}}

An argument analogous to the one used in Lemma \ref{R-L} leads us to the following expressions.

\begin{lemma}
		The quantities $R_{1,1}, R_{1,2}, R_{2,1}$ and  $R_{2,2}$ are given by
\begin{align*}			 R_{1,1}(z)=&-R_{2,2}(z)\\
=&\frac{1}{z^2-1}\left[ \frac{1}{h_{n-1}}\int_{-1}^{1}\frac{(1-x^2)v_{_J}'(x)}{x-z}P_n(x)P_{n-1}(x)w_{_J}(x)dx+nz-\mathbf{p}(n)\right],\\
R_{1,2}(z)=&\frac{1}{2\pi i}\cdot \frac{1}{z^2-1}\left[\int_{-1}^{1}\frac{(1-x^2)v_{_J}'(x)}{x-z}P_n^2(x)w_{_J}(x)dx+(2n+1)h_n\right],\\
R_{2,1}(z)=&-\frac{2 \pi i}{h_{n-1}^2}\cdot \frac{1}{z^2-1}\left[\int_{-1}^{1}\frac{(1-x^2)v_{_J}'(x)}{x-z}P_{n-1}^2(x)w_{_J}dx+(2n-1)h_{n-1}\right].
\end{align*}
\end{lemma}
Finally, by substituting the above expressions into \eqref{p'nyr}-\eqref{pn-1'y},
we arrive at the ladder operators \eqref{loJ}-\eqref{Bn-J}, from which, as noted in Remark \ref{lr-r}, the compatibility conditions $(S_1), (S_2)$ and $(S_2')$ are derived.

\subsection{An alternative derivation of the ladder operators}\label{proof-J-1}
We first derive the lowering operator \eqref{lo}.	 Noting that $P_n(z)=z^n+\textbf{p}(n)z^{n-1}+\cdots$ for $n\geq1$, we get
\begin{align}
(1-z^2)P_n'(z)=&-nz^{n+1}-(n-1)\textbf p(n)z^n+\cdots\nonumber\\
=&-nP_{n+1}(z)+\left(n \textbf p(n+1)-(n-1)\textbf p(n)\right)P_{n}(z)+\sum_{j=0}^{n-1}c_{n,j}P_j(z),\label{y2Pn'}
\end{align}
for $n\geq1$, where
	\begin{equation*}		 c_{n,j}=\frac{1}{h_j}\int_{-1}^{1}(1-x^2)P_n'(x)P_j(x)w_{_J}(x)dx,\quad\qquad 0\leq j\leq n-1.
	\end{equation*}
Integrating by parts yields
\begin{align}		 c_{n,j}=&\frac{1}{h_j}\int_{-1}^{1}(1-x^2)P_j(x)w_{_J}(x)\,dP_n(x)\nonumber\\		 =&\frac{1}{h_j}\int_{-1}^{1}\left(2xP_j(x)\right)P_n(x)w_{_J}(x)dx-\frac{1}{h_j}\int_{-1}^{1}\left((1-x^2)P_j'(x)\right)P_n(x)w_{_J}(x)dx\nonumber\\		 &+\frac{1}{h_j}\int_{-1}^{1}(1-x^2)v_{_J}'(x)P_j(x)P_n(x)w_{_J}(x)dx,\quad\qquad 0\leq j\leq n-1,\label{cnk-2}
\end{align}
where we use the fact that $\left.(1-x^2)w_{_J}(x)\right|_{x=\pm 1}=\left.(1-x)^{\alpha+1}(1+x)^{\beta+1}w_0(x)\right|_{x=\pm 1}=0$, since $\alpha,\beta>-1$ and $w_0(\pm1)$ is bounded (see Assumption \ref{ass:w0} in the Introduction). Applying the recurrence relation \eqref{3r} and equality \eqref{y2Pn'} to replace $xP_j(x)$ and $(1-x^2)P_j'(x)$ respectively in \eqref{cnk-2}, in view of the orthogonality relation \eqref{dp} and \eqref{bth}, we obtain
\begin{align*}		 c_{n,n-1}=&(n+1)\beta_n+\frac{1}{h_{n-1}}\int_{-1}^{1}(1-x^2)v_{_J}'(x)P_{n-1}(x)P_n(x)w_{_J}(x)dx,
\end{align*}
for $n\geq1$, and
\begin{align*} c_{n,j}=\frac{1}{h_j}\int_{-1}^{1}(1-x^2)v_{_J}'(x)P_j(x)P_n(x)w_{_J}(x)dx, \quad\qquad 0\leq j\leq n-2,
\end{align*}
for $n\geq2$.
Substituting them into \eqref{y2Pn'}, in light of \eqref{or-1}, we are led to
\begin{align*}
(1-z^2)P_n'(z)=&-nP_{n+1}(z)+\left(n\textbf{p}(n+1) - (n-1)\textbf{p}(n)\right)P_n(z)+(n+1)\beta_nP_{n-1}(z)\nonumber\\		 &-\int_{-1}^{1}\left((1-z^2)v_{_J}'(z)-(1-x^2)v_{_J}'(x)\right)P_n(x)w_{_J}(x)\sum_{j=0}^{n-1}\frac{P_j(z)P_j(x)}{h_j}dx.
\end{align*}
Using \eqref{3r}, \eqref{alp} and \eqref{cd} to get rid of $P_{n+1}(z)$, $\textbf p(n+1)$ and the summation term respectively, we arrive at the lowering operator \eqref{lo}.

Now we make use of the definitions of $A_n$ and $B_n$ given by \eqref{An-J}-\eqref{Bn-J} to deduce $(S_1)$, i.e.	 \[B_{n+1}(z)+B_n(z)=(z-\alpha_n)A_n(z)-v_{_J}'(z).\]
We get
\begin{align*}		 B_{n+1}(z)+B_n(z)=&\frac{1}{1-z^2}\left[\frac{1}{h_n}\int_{-1}^{1}\frac{(1-z^2)v_{_J}'(z)-(1-x^2)v_{_J}'(x)}{z-x}(x-\alpha_n)P_n^2(x)w_{_J}(x)dx\right.\nonumber\\
&\left.\qquad\qquad+(2n+1)z-\mathbf{p}(n+1)-\mathbf{p}(n)\right],		 \end{align*}	
where the identity $\frac{P_{n+1}(x)}{h_n}+\frac{P_{n-1}(x)}{h_{n-1}}=\frac{(x-\alpha_n)P_n(x)}{h_n}$, due to the recurrence relation \eqref{3r}, is used. And
\begin{align*}
(z-\alpha_n)A_n(z)=&\frac{1}{1-z^2}\left[\frac{1}{h_n}\int_{-1}^{1}\frac{(1-z^2)v_{_J}'(z)-(1-x^2)v_{_J}'(x)}{z-x}(z-\alpha_n)P_n^2(x)w_{_J}(x)dx\right.\nonumber\\
 &\left.\qquad\qquad+(2n+1)z-(2n+1)\alpha_n\right].
\end{align*}
It follows that
\begin{align}	
 &B_{n+1}(z)+B_n(z)-(z-\alpha_n)A_n(z)\nonumber\\
&=\frac{1}{1-z^2}\left[\frac{1}{h_n}\int_{-1}^{1}(1-x^2)v_{_J}'(x)P_n^2(x)w_{_J}(x)dx	 +(2n+1)\alpha_n-\mathbf{p}(n+1)-\mathbf{p}(n)\right]-v_{_J}'(z).\label{S1-1}
\end{align}
Through integration by parts, the integral term in the above square bracket is given by
\begin{align}		 &\frac{1}{h_n}\int_{-1}^{1}(1-x^2)v_{_J}'(x)P_n^2(x)w_{_J}(x)dx=-\frac{1}{h_n}\int_{-1}^{1}(1-x^2)P_n^2(x)\,dw_{_J}(x)\nonumber\\ &\qquad=-\frac{P_n^2(x)}{h_n}\left.(1-x)^{\alpha+1}(1+x)^{\beta+1}w_0(x)\right|_{x=-1}^{x=1}-\frac{2 }{h_n}\int_{-1}^{1}\left(xP_n(x)\right)P_n(x)w_{_J}(x)dx\nonumber\\
&\qquad\qquad+\frac{2}{h_n}\int_{-1}^{1}P_n(x)\left((1-x^2)P_n'(x)\right)w_{_J}(x)dx\nonumber\\
&\qquad=2\left(-\alpha_n+n\mathbf{p}(n+1)-(n-1)\mathbf{p}(n)\right),\label{intvJ'}
\end{align}
where, to obtain the last equality, we make use of the recurrence relation \eqref{3r} and  \eqref{y2Pn'}, combined with \eqref{dp}. Substituting \eqref{intvJ'} back into \eqref{S1-1}, and noting that $\textbf{p}(n+1)=\textbf{p}(n)-\alpha_n$, we come to $(S_1)$.

Combining $(S_1)$ with the lowering operator \eqref{loJ}, on account of the recurrence relation \eqref{3r}, one produces the raising operator \eqref{roJ}. See, for instance, \cite[Theorem 2.3]{MinFang25}. Using the ladder operators and $(S_1)$, with the aid of \eqref{3r} again, one obtains $(S_2)$ which combined with $(S_1)$ leads to $(S_2')$. See, e.g. \cite[Theorem 4.2]{WvA18}.

\subsection{Examples and extensions}
We provide examples to illustrate the validity of Theorem \ref{main-J} and generalize the existing results established by using the ladder operator approach from $\alpha,\beta>0$ to $\alpha,\beta>-1$. In this subsection, we assume $x\in[-1,1]$, unless stated otherwise.

\begin{exm}\label{classicJ}
 Consider the monic Jacobi polynomials orthogonal w.r.t. $w_{_J}(x)=(1-x)^{\alpha}(1+x)^{\beta},\alpha,\beta>-1$. We have
\begin{align*}
\frac{(1-z^2)v_{_J}'(z)-(1-x^2)v_{_J}'(x)}{z-x}=\alpha+\beta.
    \end{align*}
Substituting it into \eqref{An-J}-\eqref{Bn-J} gives us
\begin{align*}
A_n(z)=\frac{2n+1+\alpha+\beta}{1-z^2},\qquad\qquad
	    B_n(z)=\frac{nz-\mathbf{p}(n)}{1-z^2}.
\end{align*}
Plugging them into $(S_2')$ and multiplying both sides of the resulting equation by $(1-z^2)^2$ yields
\begin{equation*}
\begin{aligned} \left((n+\alpha+\beta)z-\mathbf{p}(n)+\alpha-\beta\right)\cdot
\left(nz-\mathbf{p}(n)\right)&+n\left(n+\alpha+\beta\right)\left(1-z^2\right)\\
&=\left(2n-1+\alpha+\beta\right)\left(2n+1+\alpha+\beta\right)\beta_n.
\end{aligned}
\end{equation*}
Setting $z=\pm1$ in the above equality results in
	\begin{align}		 \left(\mathbf{p}(n)-n-2\alpha\right)\left(\mathbf{p}(n)-n\right)
=&\left(2n-1+\alpha+\beta\right)\left(2n+1+\alpha+\beta\right)\beta_n\label{pbt}\\
=&\left(\mathbf{p}(n)+n+2\beta\right)\left(\mathbf{p}(n)+n\right).\label{p}
\end{align}
It follows from \eqref{p} that
	\begin{equation*}	 \mathbf{p}(n)=\frac{n(\alpha-\beta)}{2n+\alpha+\beta}.
	\end{equation*}	
Inserting it into \eqref{alp}, i.e. $\alpha_n=\mathbf{p}(n)-\mathbf{p}(n+1)$, and \eqref{pbt} leads us to
	\begin{align}		 \alpha_n=&\frac{\beta^2-\alpha^2}{(2n+\alpha+\beta)(2n+2+\alpha+\beta)},\label{al-J}\\
\beta_n=&\frac{4n(n+\alpha)(n+\beta)(n+\alpha+\beta)}{(2n+\alpha+\beta)^2(2n+1+\alpha+\beta)(2n-1+\alpha+\beta)}.\label{bt-J}
\end{align}
\end{exm}

\begin{remark}
\eqref{al-J}-\eqref{bt-J} were deduced in \cite{ChenIsmail05} for the monic Jacobi polynomials with $\alpha,\beta>0$ by using $(S_1)$ and $(S_2)$. The computations therein involve additional technical steps, compared with our direct formulation presented above.

In addition, by a change of variable, we can obtain from \eqref{al-J}-\eqref{bt-J} the recurrence coefficients for the shifted monic Jacobi polynomials $\widehat{P}_n^{(\alpha,\beta)}(x)$ orthogonal w.r.t. $x^{\alpha}(1-x)^{\beta},x\in[0,1]$, namely
\vspace{-2mm}
\[\int_0^1\widehat{P}_n^{(\alpha,\beta)}(x)\widehat{P}_m^{(\alpha,\beta)}(x)x^{\alpha}(1-x)^{\beta}dx= h_n\delta_{mn},\]
where $\delta_{mn}$ is the Kronecker delta function. Setting $x=(y+1)/2$ in the above integral and writing $\widehat{P}_n^{(\alpha,\beta)}((y+1)/2)=:P_n^{(\beta,\alpha)}(y)/2^n$, we get
\[\int_{-1}^1P_n^{(\beta,\alpha)}(y)P_m^{(\beta,\alpha)}(y)(1-y)^{\beta}(1+y)^{\alpha}dy=2^{m+n+\alpha+\beta+1}h_n\delta_{mn},\]
which indicates that $\{P_n^{(\beta,\alpha)}(y)\}$ are essentially the monic Jacobi polynomials orthogonal w.r.t. $(1-y)^{\beta}(1+y)^{\alpha}, y\in[-1,1]$. According to Example \ref{classicJ}, we have
\[yP_n^{(\beta,\alpha)}(y)=P_{n+1}^{(\beta,\alpha)}(y)-\alpha_nP_n^{(\beta,\alpha)}(y)+\beta_nP_{n-1}^{(\beta,\alpha)}(y)\]
where $\alpha_n$ and $\beta_n$ are given by \eqref{al-J}-\eqref{bt-J}. Here the negative sign before $\alpha_n$ is due to the exchange of $\alpha$ and $\beta$ in \eqref{al-J}. Replacing $P_j^{(\beta,\alpha)}(y)$ by $2^j\widehat{P}_j^{(\alpha,\beta)}(x)$ for $j=n-1,n,n+1$ and $y$ by $2x-1$ in the above identity, dividing both sides of the resulting equation by $2^{n+1}$, we come to the recurrence relation satisfied by $\widehat{P}_n^{(\alpha,\beta)}(x):$
\begin{align}\label{cJ-sJ-3r}
x\widehat{P}_n^{(\alpha,\beta)}(x)=\widehat{P}_{n+1}^{(\alpha,\beta)}(x)+\frac{1-\alpha_n}{2}\widehat{P}_n^{(\alpha,\beta)}(x)+\frac{\beta_n}{4}\widehat{P}_{n-1}^{(\alpha,\beta)}(x).
\end{align}
\end{remark}

We now first make use of \eqref{An-J}-\eqref{Bn-J} to compute $A_n$ and $B_n$ for several Jacobi-type weight functions which were investigated in earlier research under the assumption that $\alpha,\beta>0$. Then, for $w_{_J}(x)$ that is not even, we differentiate the orthogonality relations satisfied by the associated monic orthogonal polynomials, namely
 \begin{gather*}	
 \int_{-1}^{1}P_n^2(x)w_{_J}(x)dx=h_n,\\
 \int_{-1}^{1}P_n(x)P_{n-1}(x)w_{_J}(x)dx=0.
\end{gather*}
When $w_{_J}(x)$ is even, the second identity should be replaced by
\[\int_{-1}^{1}P_n(x)P_{n-2}(x)w_{_J}(x)dx=0.\label{Pnn-2-J}\]
 We compare the resulting expressions of $\{A_n,B_n\}$ and the differential equalities for $\alpha,\beta>-1$ with previous work for $\alpha,\beta>0$. If they are consistent, following the argument presented just before Example \ref{ex-2-L}, then we can generalize the existing findings for the associated Hankel determinants to the range $\alpha,\beta>-1$. We first analyze three families of semi-classical weights.
 \\
\begin{exm}\label{ex-J2}
The Hankel determinant for the weight function $w_{_J}(x)=(1-x)^{\alpha}(1+x)^{\beta}\e^{-tx},\alpha,\beta>0, t\in\mathbb{R}$ was studied in \cite{BCE10}. Supposing $\alpha,\beta>-1$, by inserting
\[\frac{(1-z^2)v_{_J}'(z)-(1-x^2)v_{_J}'(x)}{z-x}=-t(x+z)+\alpha+\beta
\]
into \eqref{An-J}-\eqref{Bn-J}, we get
\begin{align*}
A_n(z)=&\frac{-t(\alpha_n+z)+2n+1+\alpha+\beta}{1-z^2}=:\frac{R_n(t)}{1-z}+\frac{R_n(t)+t}{1+z},\\ B_n(z)=&\frac{-t\beta_n+nz-\mathbf{p}(n)}{1-z^2}=:\frac{r_n(t)}{1-z}+\frac{r_n(t)-n}{1+z},
\end{align*}
where
\[
R_n(t):=\frac{2n+1+\alpha+\beta-t\alpha_n-t}{2},\qquad\qquad r_n(t):=\frac{n-t\beta_n-\mathbf{p}(n)}{2}.\]
We observe that our expressions of $\{A_n, B_n\}$ are formally identical to those in \cite{BCE10} where $\alpha,\beta>0$. Moreover, when $\alpha,\beta>0$, according to (3.13) and the identity between (3.20)-(3.21) in \cite{BCE10}, which were deduced by substituting their expressions of $A_n$ and $B_n$ into $(S_1)$ and $(S_2')$, we find that
\begin{align*}
R_n(t)=\frac{\alpha}{h_n}\int_{-1}^1\frac{P_n^2(y)}{1-y}w_{_J}(y)dy,\qquad\qquad r_n(t)=\frac{\alpha}{h_n}\int_{-1}^1\frac{P_n(y)P_{n-1}(y)}{1-y}w_{_J}(y)dy.
\end{align*}
In addition, we readily see that the differential identities (2.9) and (2.12) in \cite{BCE10} remain valid for $\alpha,\beta>-1$. Therefore, the conclusion that the Hankel determinant was closely connected to a particular Painlev\'{e} V equation, deduced in \cite{BCE10} for $\alpha,\beta>0$, can be generalized to $\alpha,\beta>-1$.
\end{exm}

\begin{exm}\label{ex-J3}
The symmetric weight function $w_{_J}(x)=(1-x^2)^{\alpha}\e^{-tx^2},\alpha>0,t\in\mathbb{R}$ was considered in \cite{MinChen22}. The associated monic orthogonal polynomials read
\[P_n(x):=x^n+q(n,t)x^{n-2}+\cdots, \quad\qquad n=2,3,\cdots,\]
with $P_0(x):=1$ and $P_1(x):=x$. Throughout this example, unless necessary, we omit the $t$-dependence of $P_n$ for notational simplicity. The following recurrence relation holds
\begin{align*}
xP_{n}(x)=P_{n+1}(x)+\beta_nP_{n-1}(x),
\end{align*}
for $n\geq0$, with $P_{-1}(x):=0$.
Using it and the orthogonality relation \eqref{dp}, we get
\begin{gather}
\frac{1}{h_n}\int_{-1}^1x^2P_n^2(x)w_{_J}(x)dx=\beta_{n+1}+\beta_{n},\label{ex8-1}\\
\frac{1}{h_{n-1}}\int_{-1}^1xP_n(x)P_{n-1}(x)w_{_J}(x)dx=\beta_n.\label{ex8-2}
\end{gather}
According to the facts that $P_n(-x)=(-1)^nP_n(x)$ and $w_{_J}(x)$ is even, in view of parity, we find
\begin{align}\label{ex8-3}
\frac{1}{h_n}\int_{-1}^1xP_n^2(x)w_{_J}(x)dx=0=
\frac{1}{h_{n-1}}\int_{-1}^1x^2P_n(x)P_{n-1}(x)w_{_J}(x)dx.
\end{align}
Assuming $\alpha>-1$, by plugging
\[\frac{(1-z^2)v_{_J}'(z)-(1-x^2)v_{_J}'(x)}{z-x}=-2t(x^2+zx+z^2-1)+2\alpha
\]
and $\mathbf{p}(n)=0$ into \eqref{An-J}-\eqref{Bn-J}, with the help of \eqref{ex8-1}-\eqref{ex8-3}, we obtain
\begin{equation}\label{AnBn-J-3}
\begin{aligned}
A_n(z)=2t+\frac{R_n(t)}{1-z^2},\qquad\qquad B_n(z)=\frac{z\,r_n(t)}{1-z^2},
\end{aligned}
\end{equation}
where
\begin{align}\label{defRnrn-ex8}
R_n(t):=2n+1+2\alpha-2t(\beta_{n+1}+\beta_n),\qquad\qquad r_n(t):=n-2t\beta_n.
\end{align}
We find that our expressions of $\{A_n, B_n\}$ are identical in form structure to the ones in \cite{MinChen22} where $\alpha>0$. Moreover, when $\alpha>0$, from (2.14) and (2.18) in \cite{MinChen22} which were obtained by plugging the expressions of $A_n$ and $B_n$ into $(S_1)$ and $(S_2')$, it follows that
\begin{align*}
R_n(t)=\frac{2\alpha}{h_n}\int_{-1}^1\frac{P_n^2(x)w_{_J}(x)dx}{1-x^2},\qquad\qquad r_n(t)=\frac{2\alpha}{h_{n-1}}\int_{-1}^1\frac{x}{1-x^2}P_n(x)P_{n-1}(x)w_{_J}(x)dx.
\end{align*}

To continue, we differentiate both sides of the orthogonality relations
\begin{gather*}
h_n(t)=\int_{-1}^1P_n^2(x;t)(1-x^2)^{\alpha}\e^{-tx^2}dx,\\
0=\int_{-1}^1P_n(x;t)P_{n-2}(x;t)(1-x^2)^{\alpha}\e^{-tx^2}dx,
\end{gather*}
w.r.t. $t$. According to \eqref{ex8-1} and the equality
\[\frac{1}{h_{n-2}}\int_{-1}^1x^2P_n(x;t)P_{n-2}(x;t)(1-x^2)^{\alpha}\e^{-tx^2}dx=\beta_n\beta_{n-1},\]
which results from the recurrence and orthogonality relations, we get
\begin{align}
2t\frac{d}{dt}\ln h_n(t)=&-2t(\beta_{n+1}+\beta_n)=R_n(t)-2n-1-2\alpha,\label{exJ-8-1}\\
2t\frac{d}{dt}q(n,t)=&2t\beta_n\beta_{n-1}.\label{exJ-8-2}
\end{align}
Equation \eqref{exJ-8-1} is the same as (4.1) of \cite{MinChen22} where $\alpha>0$, while \eqref{exJ-8-2} does not align with (4.2) therein since an integration by parts valid for $\alpha>0$ was applied.
However, we can derive their (4.2), i.e.
 \begin{align*}
 2t\frac{d}{dt}q(n,t)=\beta_nR_n-r_n-2q(n,t),
 \end{align*}
 from a different perspective. According to \eqref{exJ-8-2}, it suffices to show that
 \begin{align}\label{ex8-2}
2t\beta_n\beta_{n-1}=\beta_nR_n-r_n-2q(n,t).
\end{align}
From the definition of $r_n$ given in \eqref{defRnrn-ex8}, we have
 \begin{align}\label{exJ-8-3}
 \beta_n=\frac{n-r_n}{2t},
 \end{align}
so that
\begin{align}\label{exJ-8-4}
2t\beta_n\beta_{n-1}=\frac{1}{2t}(n-r_n)(n-1-r_{n-1}).
\end{align}
Inserting \eqref{AnBn-J-3} into $(S_1)$ and $(S_2')$, and by equating the coefficients of $1/(1-z^2)$ on both sides of the resulting equations, we find
\begin{gather}
r_{n+1}+r_n=R_n-2\alpha,\label{3-1}\\
2t\beta_nR_{n-1}=2(t-\alpha)r_n-r_n^2+\sum_{j=0}^{n-1}R_j-2t\beta_nR_n.\label{3-2}
\end{gather}
Using \eqref{3-1} to get rid of $r_{n-1}$ in \eqref{exJ-8-4}, in view of \eqref{exJ-8-3}, we are led to
\begin{align}\label{exJ-8-5}
2t\beta_n\beta_{n-1}=\frac{1}{2t}\left[(n-r_n)(n+2\alpha-1+r_n)- 2t\beta_nR_{n-1}\right].
\end{align}
From the definition of $R_j$ given by \eqref{defRnrn-ex8}, on account of $\sum\limits_{j=0}^{n-1}\beta_j=-q(n,t)$ and, once again, \eqref{exJ-8-3}, we obtain \begin{align}\label{3-4}
\sum_{j=0}^{n-1}R_j=n(n+2\alpha-1)+4tq(n,t)+r_n.
\end{align}
Substituting it into \eqref{3-2} and plugging the obtained expression for $2t\beta_nR_{n-1}$ into \eqref{exJ-8-5}, we finally arrive at \eqref{ex8-2}.
\\
\\

Therefore, the second order difference equations for $q(n,t)$, $\beta_n$ and the logarithmic derivative of the Hankel determinant, the second order differential equations satisfied by the latter two, and the Painlev\'{e} V equation for $R_n$, which were established in \cite{MinChen22} and valid for $\alpha>0$, can be generalized to $\alpha>-1$.
\end{exm}

\begin{exm}
The symmetric weight function $w_{_J}(x)=\left(1-x^2\right)^{\alpha}\left(1-k^2x^2\right)^{\gamma},\alpha>0,\gamma\in\mathbb{R}$ was investigated in \cite{BCH15}. One has $\mathbf{p}(n)=0$ and
\begin{align*}
\int_{-1}^1\frac{P_n^2(x)w_{_J}(x)dx}{x-\frac{1}{k}}=&-\int_{-1}^1\frac{P_n^2(x)w_{_J}(x)dx}{x+\frac{1}{k}},\\
\int_{-1}^1\frac{P_n(x)P_{n-1}(x)w_{_J}(x)dx}{x-\frac{1}{k}}=&\int_{-1}^1\frac{P_n(x)P_{n-1}(x)w_{_J}(x)dx}{x+\frac{1}{k}},
\end{align*}
which are due to the facts that $P_n(-x)=(-1)^nP_n(x)$ and $w_{_J}(x)$ is even. See also (2.5)-(2.6) of \cite{BCH15}.
Assuming $\alpha>-1$, by substituting
\[\frac{(1-z^2)v_{_J}'(z)-(1-x^2)v_{_J}'(x)}{z-x}=2(\alpha+\gamma)+\gamma\left(1-\frac{1}{k^2}\right)\left[\frac{1}{(z-\frac{1}{k})(x-\frac{1}{k})}+\frac{1}{(z+\frac{1}{k})(x+\frac{1}{k})}\right]
\]
into \eqref{An-J}-\eqref{Bn-J}, we get
\begin{align*}
A_n(z)=\frac{2(n+\alpha+\gamma)+1-\frac{2}{k}R_n^*}{1-z^2}+\frac{-\frac{2}{k}R_n^{*}}{z^2-\frac{1}{k^2}},\qquad\qquad B_n(z)=\frac{z(n+2r_n^*)}{1-z^2}+\frac{2zr_n^{*}}{z^2-\frac{1}{k^2}},
\end{align*}
where
\begin{align*}
R_n^{*}(k^2):=\frac{\gamma}{h_n}\int_{-1}^1\frac{P_n^2(x)w_{_J}(x)dx}{x+\frac{1}{k}},\qquad\qquad r_n^{*}(k^2):=\frac{\gamma}{h_{n-1}}\int_{-1}^1\frac{P_n(x)P_{n-1}(x)w_{_J}(x)dx}{x+\frac{1}{k}}.
\end{align*}
When $\alpha>0$, from (2.22)-(2.23) in \cite{BCH15} which were obtained by using $(S_1)$ and $(S_2)$, we find that their expressions of $A_n$ and $B_n$ are consistent with ours. Therefore, the second order difference equations established in \cite{BCH15} for $\beta_n$ and the coefficient of $x^{n-2}$ in $P_n(x)$ still hold for $\alpha>-1$. They can also be deduced by inserting our expressions of $A_n$ and $B_n$ into $(S_1), (S_2)$ and $(S_2')$, and we believe that the derivation is simpler than that in \cite{BCH15} where two additional auxiliary quantities were introduced alongside $R_n^*$ and $r_n^*$. In addition, as pointed out in \cite{BCH15}, the Hankel determinant generated by $w_{_J}(x)$ with $\alpha>0$ was connected with the one which will be discussed in Example \ref{exJS-2}, and its logarithmic derivative was expressed in terms of two functions satisfying the $\sigma$-form of Painlev\'{e} VI. This result can be generalized to $\alpha>-1$, based on the analysis in Example \ref{exJS-2}.
\end{exm}

We now focus on two classes of Jacobi-type orthogonal polynomials associated with symmetric weight functions involving $\e^{-t/x^2}$ or $\e^{-t/(1-x^2)}$.
\\
\\
\begin{exm}
The weight function $w_{_J}(x)=(1-x^2)^{\alpha}\e^{-t/x^2}, \alpha>0,t\geq0$ was analyzed in \cite{MinChenNPB20}. One knows that $P_n(x)$ contains only odd or even powers of $x$ when $n$ is odd or even respectively. It follows that
\vspace{-1mm}
\[\frac{1}{h_{n-1}}\int_{-1}^1 \frac{P_n(x)P_{n-1}(x)}{x}w_{_J}(x)dx=\frac{1-(-1)^n}{2}.\]
See (2.9) of \cite{MinChenNPB20}. Supposing $\alpha>-1$, by inserting
\begin{align*}
\frac{(1-z^2)v_{_J}'(z)-(1-x^2)v_{_J}'(x)}{z-x}=2\alpha+2t\left[\frac{1}{zx^3}+\frac{1}{z^2x^2}+\left(\frac{1}{z^3}-\frac{1}{z}\right)\cdot\frac{1}{x}\right]
\end{align*}
and $\mathbf{p}(n)=0$ into \eqref{An-J}-\eqref{Bn-J}, in view of the parity, we get the expressions of $A_n$ and $B_n$ which are identical to (2.5) and (2.6) of \cite{MinChenNPB20}. The differential relations given by (3.1) and (3.5) therein, valid for $\alpha>0$, continue to hold for $\alpha>-1$.

The weight function  $(1-x^2)^{\alpha}\e^{-t/(1-x^2)}, \alpha,t>0$ was considered in \cite{MinChen21}. Following a similar argument as above, we verify that the expressions of $A_n$ and $B_n$ given by (21)-(22) in \cite{MinChen21} and the differential relations (53)-(54) therein remain valid for $\alpha>-1$.

\sloppy{Therefore, the finite-dimensional results established in \cite{MinChenNPB20} and \cite{MinChen21} with $\alpha>0$ can be extended to $\alpha>-1$, including the second order difference equation and differential equation satisfied by the logarithmic derivative of the Hankel determinant, its integral representation in terms of Painlev\'{e} V transcendent, etc. In addition, as stated in \cite{MinChenNPB20} and \cite{MinChen21}, the Hankel determinants generated by the above two weights with $\alpha>0$ were connected with the one which will be discussed in Example \ref{exJS-1}, and their logarithmic derivatives were expressed in terms of two functions satisfying the $\sigma$-form of Painlev\'{e} V. These conclusions can also be generalized to $\alpha>-1$, based on the analysis in Example \ref{exJS-1}.}
\end{exm}

We close this section by presenting the ladder operators for the monic Jacobi-type orthogonal polynomials with jump discontinuities. They can be derived by using the argumentation strategy from the preceding two subsections,
\begin{theorem}
Denote the Jacobi-type weight function with jump discontinuities, i.e.
\begin{align}\label{ex7-1}
\widetilde{w}(x; \vec{t}\,):= w_{_J}(x) \left( \omega_0 + \sum_{k=1}^{m} \omega_k \theta(x - t_k) \right),
\end{align}
where $w_{_J}$ is defined by \eqref{wJ} and independent of $\vec{t}=(t_1,\cdots,t_m)$ with $-1\leq t_1<\cdots<t_m\leq1$, $\{\omega_k, 0\leq k\leq m\}$ and $\theta(\cdot)$ have the same meanings as in Theorem \ref{L-jump}.  The associated monic orthogonal polynomials satisfy the ladder operators \eqref{loJ}-\eqref{roJ} with $A_n(z)$ and $B_n(z)$ given by
\\
\\
\begin{subequations}\label{AnBn-JJ}
\begin{equation}			
\begin{aligned}	 A_n(z)=&\frac{1}{1-z^2}\left[\frac{1}{h_n}\int_{-1}^{1}\frac{(1-z^2)v_{_J}'(z)-(1-x^2)v_{_J}'(x)}{z-x}P_n^2(x; \vec{t}\,)\widetilde{w}(x; \vec{t}\,)dx\right.\\
&\left.\qquad\qquad+2n+1+\sum_{k=1}^{m}R_{n,k}^{\star}(z+t_k)\right]+\sum_{k=1}^{m}\frac{R^{\star}_{n,k}}{z-t_k}, \end{aligned}
\end{equation}
\begin{equation}
\begin{aligned}
B_n(z)=&\frac{1}{1-z^2}\left[ \frac{1}{h_{n-1}}\int_{-1}^{1}\frac{(1-z^2)v_{_J}'(z)-(1-x^2)v_{_J}'(x)}{z-x}P_n(x; \vec{t}\,)P_{n-1}(x; \vec{t}\,)\widetilde{w}(x; \vec{t}\,)dx\right.\\
&\left.\qquad\qquad+nz-\mathbf{p}(n)+\sum_{k=1}^{m}r_{n,k}^{\star}(z+t_k)\right]+\sum_{k=1}^{m}\frac{r^{\star}_{n,k}}{z-t_k},
\end{aligned}
\end{equation}
\end{subequations}
where $v_{_J}(x):=-\ln w_{_J}(x)$ and $\{R_{n,k}^{\star},r_{n,k}^{\star}, k=1,\cdots,m\}$ are defined by
\begin{equation*}		
R^{\star}_{n,k}( \vec{t}\,):=\frac{\omega_kw_{_J}(t_k)P_n^2(t_k; \vec{t}\,)}{h_n},\qquad\qquad r^{\star}_{n,k}( \vec{t}\,):=\frac{\omega_kw_{_J}(t_k)P_n(t_k; \vec{t}\,)P_{n-1}(t_k; \vec{t}\,)}{h_{n-1}}.
\end{equation*}
\end{theorem}

\begin{exm}
The probability that the interval $(-a,a), a\in(0,1),$ contains no eigenvalues of the Jacobi unitary ensemble associated with $(1-x^2)^{\alpha},\alpha>0,$ was studied in \cite{MinChen2017}. Up to a constant term, it can be evaluated as the Hankel determinant generated by the weight function $(1-x^2)^{\alpha}(1-\theta(x+a)+\theta(x-a))$ which is a special case of \eqref{ex7-1} with $w_{_J}(x)=(1-x^2)^{\alpha}, m=2, \omega_0=1,\omega_1=-1,$ $\omega_2=1, t_1=-a,t_2=a$. In this case, we have
\[\frac{(1-z^2)v_{_J}'(z)-(1-x^2)v_{_J}'(x)}{z-x}=2\alpha,\]
and
\[\mathbf{p}(n)=0,\qquad\quad R_{n,1}^{\star}(a)=-R_{n,2}^{\star}(a),\qquad\quad r_{n,1}^{\star}(a)=r_{n,2}^{\star}(a),\]
with
\[R_{n,2}^{\star}(a)=\frac{(1-a^2)^{\alpha}P_n^2(a;a)}{h_n},\qquad\qquad r_{n,2}^{\star}(a)=\frac{(1-a^2)^{\alpha}P_n(a;a)P_{n-1}(a;a)}{h_{n-1}}.\]
Hence, \eqref{AnBn-JJ} now read
\begin{align*}
A_n(z)=\frac{2n+2\alpha+1+2aR_{n,2}^{\star}}{1-z^2}+\frac{2aR_{n,2}^{\star}}{z^2-a^2},\qquad\quad B_n(z)=\frac{z\left(n+2r_{n,2}^{\star}\right)}{1-z^2}+\frac{2zr_{n,2}^{\star}}{z^2-a^2}.
\end{align*}
In \cite{MinChen2017} where $\alpha>0$, four auxiliary quantities were introduced instead of two. Their expressions of $A_n$ and $B_n$ given in Theorem 1, combined with (4)-(5) therein which were obtained through integration by parts and using $(S_1)$, agree with ours. The differential relations (19)-(20) in their work remain valid for $\alpha>-1$. Consequently, the second order differential equation deduced in \cite{MinChen2017} (c.f. (38)) for the logarithmic derivative of the gap probability with $\alpha>0$ continues to hold when $\alpha>-1$.
\end{exm}

\section{Ladder Operators for shifted Jacobi-type Orthogonal Polynomials with $\alpha,\beta>-1$}
We now focus on the monic shifted Jacobi-type orthogonal polynomials associated with \eqref{wJS}, namely,
\begin{align*}
w_{_{JS}}(x)=&x^{\alpha}(1-x)^{\beta}w_0(x),\qquad x\in[0,1],\quad\alpha,\beta>-1,
\end{align*}
where, as assumed in Introduction, $w_0(x)$ is continuously differentiable and bounded at $x=0,1$.
By adapting the derivations in Sections \ref{proof-J-RHP} and \ref{proof-J-1} to this case, we obtain the ladder operators.
\begin{theorem}\label{main-JS}
The monic shifted Jacobi-type orthogonal polynomials associated with $w_{_{JS}}(z)$ satisfy the following ladder operators:
\begin{align}
P_n'(z)=&-B_n(z)P_n(z)+\beta_nA_n(z)P_{n-1}(z),\label{loJS}\\
P_{n-1}'(z)=&\left(B_n(z)+v_{_{JS}}'(z)\right)P_{n-1}(z)-A_{n-1}P_n(z),\label{roJS}
\end{align}	
for $n\geq0$, where $v_{_{JS}}(z):=-\ln w_{_{JS}}(z)$. The quantities $A_n(z)$ and $B_n(z)$ are given by
\begin{equation}\label{An-JS}		 A_n(z):=\frac{1}{z-z^2}\left[\frac{1}{h_n}\int_{0}^{1}\frac{(z-z^2)v_{_{JS}}'(z)-(x-x^2)v_{_{JS}}'(x)}{z-x}P_n^2(x)w_{_{JS}}(x)dx+2n+1\right],
\end{equation}
for $n\geq0$ with $A_{-1}(z):=0$, and
\begin{equation}\label{Bn-JS}
\begin{aligned}		 B_n(z):=&\frac{1}{z-z^2}\left[\frac{1}{h_{n-1}}\int_{0}^{1}\frac{(z-z^2)v_{_{JS}}'(z)-(x-x^2)v_{_{JS}}'(x)}{z-x}P_n(x)P_{n-1}(x)w_{_{JS}}(x)dx\right.\\
&\left.\qquad\qquad+n(z-1)-\mathbf{p}(n)\right],
\end{aligned}
\end{equation}
for $n\geq1$ with $B_0(z):=0$.
Moreover, $A_n(z)$ and $B_n(z)$ satisfy $(S_1), (S_2)$ and $(S_2')$.
\end{theorem}

\begin{remark}
When $\alpha,\beta>0$, using reasoning analogous to Remark \ref{re-J-1}, we find that \eqref{An-JS}-\eqref{Bn-JS} agree with \eqref{a-n}-\eqref{b-n}.
\end{remark}	

We construct the above ladder operators through the RHP, relying on the following lemma.
\begin{lemma} We have
\begin{align}		 C(P_{n-1}P_nw_{_{JS}})'=&C(P_{n-1}'P_nw_{_{JS}})+C(P_{n-1}P_n'w_{_{JS}})\nonumber\\		 +\frac{1}{z^2-z}&\left[\int_{0}^{1}\frac{(x-x^2)v_{_{JS}}'(x)}{x-z}P_{n-1}(x)P_n(x)w_{_{JS}}(x)dx+\left(n(z-1)-\mathbf{p}(n)\right)h_{n-1}\right],\label{C'-JS-1}\\
C(P_n^2w_{_{JS}})'=2C&(P_nP_n'w)+\frac{1}{z^2-z}\left[\int_{0}^{1}\frac{(x-x^2)v_{_{JS}}'(x)}{x-z}P_n^2(x)w_{_{JS}}(x)dx+(2n+1)h_n\right],\nonumber
\end{align}
and the four elements of $R(z):=Y'(z)Y^{-1}(z)$ read
\begin{align*}		
R_{1,1}(z)=&-R_{2,2}(z)\\
=&\frac{1}{z^2-z}\left[\frac{1}{h_{n-1}}\int_{0}^{1}\frac{(x-x^2)v_{_{JS}}'(x)}{x-z}P_{n-1}(x)P_n(x)w_{_{JS}}(x)dx+n(z-1)-\mathbf{p}(n)\right],\\
R_{1,2}(z)=&\frac{1}{2\pi i}\cdot \frac{1}{z^2-z}\left[\int_{0}^{1}\frac{(x-x^2)v_{_{JS}}'(x)}{x-z}P_n^2(x)w_{_{JS}}(x)dx+(2n+1)h_n\right],\\
R_{2,1}(z)=&\frac{-2\pi i}{h_{n-1}^2}\cdot\frac{1}{z^2-z}\left[\int_{0}^{1}\frac{(x-x^2)v_{_{JS}}'(x)}{x-z}P_{n-1}^2(x)w_{_{JS}}(x)dx+(2n-1)h_{n-1}\right].
\end{align*}
\end{lemma}	
\begin{proof} By definition of $C(P_{n-1}P_nw_{_{JS}})$, we get
\begin{align*}
C(P_{n-1}P_nw_{_{JS}})' &=\frac{d}{dz}\int_{0}^{1}\frac{P_{n-1}(x)P_n(x)w_{_{JS}}(x)}{x-z}dx =\int_{0}^{1}\frac{P_{n-1}(x)P_n(x)w_{_{JS}}(x)}{(x-z)^2}dx \\		 &=\int_{0}^{1}P_{n-1}(x)P_n(x)w_{_{JS}}(x)\left(\frac{1}{(x-z)^2}+\frac{1}{z-z^2}\right)dx \\		 &=-\int_{0}^{1}P_{n-1}(x)P_n(x)w_{_{JS}}(x)\,d\left(\frac{1}{x-z}+\frac{x+z-1}{z^2-z}\right)\\
&=-\int_{0}^{1}P_{n-1}(x)P_n(x)w_{_{JS}}(x)\,d\left(\frac{x^2-x}{(x-z)(z^2-z)}\right).
\end{align*}
Through integration by parts and in view of the orthogonality relations \eqref{dp}-\eqref{or-1}, we obtain \eqref{C'-JS-1}. The detailed derivation of it, along with the remaining proof of this lemma, parallels that in Lemma \ref{lemma-J} and Lemma \ref{R-L}, and is thus omitted.
\end{proof}

Substituting the expressions of $\{R_{1,1},R_{1,2},R_{2,1},R_{2,2}\}$ into \eqref{p'nyr}-\eqref{pn-1'y} leads us to Theorem \ref{main-JS}. The ladder operators can also be deduced by writing
\begin{equation}\label{lo-JS-1}		 (z-z^2)P_n'(z)=-nP_{n+1}(z)+\left(n+n\textbf{p}(n+1)-(n-1)\textbf{p}(n)\right)P_n(z)+\sum_{k=0}^{n-1}\widehat{c}_{n,k}P_k(z),
	\end{equation}
where
\begin{equation*}	 \widehat{c}_{n,k}=\frac{1}{h_k}\int_{0}^{1}(x-x^2)P_n'(x)P_k(x)w_{_{JS}}(x)dx,\qquad k=0,1,\cdots,n-1.
\end{equation*}
The subsequent proof can be carried out by arguments similar to those in Section \ref{proof-J-1}.

\begin{exm}
Consider the monic shifted Jacobi polynomials orthogonal w.r.t. $w_{_{JS}}(x)=x^\alpha(1-x)^\beta,\alpha,\beta>-1,x \in [0,1]$. We get
	\begin{equation}		 \frac{(z-z^2)v_{JS}'(z)-(x-x^2)v_{JS}'(x)}{z-x}=\alpha+\beta.\nonumber
	\end{equation}
Inserting it into (5.3)-(5.4) yields
	\begin{equation}		 A_n(z)=\frac{2n+1+\alpha+\beta}{z-z^2},\qquad\qquad\quad B_n(z)=\frac{n(z-1)-\mathbf{p}(n)}{z-z^2}.\nonumber
	\end{equation}
Substituting them into ($S_2'$) and multiplying both sides of the obtained equation by $(z-z^2)^2$ gives us
	\begin{align}		 \left(\mathbf{p}(n)+n(1-z)\right)&\left(\mathbf{p}(n)+(n+\alpha+\beta)(1-z)-\beta\right)+n(n+\alpha+\beta)\cdot z(1-z)\nonumber\\		 &\qquad\qquad=(2n-1+\alpha+\beta)(2n+1+\alpha+\beta)\beta_n .\nonumber
	\end{align}
	Setting $z=0,1$ in the above identity leads us to
	\begin{align}
		\label{z=0}		 \left(\mathbf{p}(n)+n\right)\left(\mathbf{p}(n)+n+\alpha\right)&=(2n-1+\alpha+\beta)(2n+1+\alpha+\beta)\beta_n\\		 \label{z=1}		 &=\mathbf{p}(n)\left(\mathbf{p}(n)-\beta\right).\nonumber
	\end{align}
	It follows from the second identity that
	\begin{equation}
		 \mathbf{p}(n)=-\frac{n(n+\alpha)}{2n+\alpha+\beta}\nonumber.
	\end{equation}
Plugging it into (1.9), i.e. $\alpha_n=\mathbf{p}(n)-\mathbf{p}(n+1)$, and \eqref{z=0} results in
	\begin{align*}		 \alpha_n&=\frac{2n(n+\alpha+\beta+1)+(\alpha+\beta)(\alpha+1)}{(2n+\alpha+\beta)(2n+2+\alpha+\beta)},\\		 \beta_n&=\frac{n(n+\alpha)(n+\beta)(n+\alpha+\beta)}{(2n+\alpha+\beta)^2(2n+1+\alpha+\beta)(2n-1+\alpha+\beta)},
\end{align*}
which agree with \eqref{cJ-sJ-3r}.
\end{exm}

 We now make use of \eqref{An-JS}-\eqref{Bn-JS} to compute $A_n$ and $B_n$ for two classes of shifted Jacobi-type orthogonal polynomials and differentiate the orthogonality relations
 \begin{gather}	
 \int_{0}^{1}P_n^2(x)w_{_{JS}}(x)dx=h_n,\label{or-Pn2-JS}\\
 \int_{0}^{1}P_n(x)P_{n-1}(x)w_{_{JS}}(x)dx=0.\label{or-Pnn-1-JS}
\end{gather}
Following the reasoning presented just before Example \ref{ex-2-L}, if the obtained expressions of $\{A_n,B_n\}$ and the differential identities are consistent with those for $\alpha,\beta>0$, then the finite-dimensional results established in prior work by using the ladder operator framework, valid for $\alpha,\beta>0$, can be generalized to $\alpha,\beta>-1$.
	
\begin{exm}\label{exJS-1}
The Pollaczek-Jacobi weight $w_{_{JS}}(x)=x^{\alpha}(1-x)^{\beta}\e^{-t/x},x\in[0,1],\alpha,\beta>0$, $t\geq0$ was studied in \cite{CD10}. Assuming $\alpha,\beta>-1$, by plugging
\begin{align*}
\frac{(z-z^2)v_{_{JS}}'(z)-(x-x^2)v_{_{JS}}'(x)}{z-x}=\frac{t}{xz}+\alpha+\beta
\end{align*}
into \eqref{An-JS}-\eqref{Bn-JS}, we get
\begin{align*}
A_n(z)=\frac{R_n^*}{z^2}+\frac{2n+1+\alpha+\beta+R_n^*}{z(1-z)},\qquad\qquad B_n(z)=\frac{r_n^*}{z^2}+\frac{r_n^*-\mathbf{p}(n)-n}{z}+\frac{r_n^*-\mathbf{p}(n)}{1-z},
\end{align*}
where
\begin{align*}
&R_n^*(t):=\frac{t}{h_n}\int_{0}^1 P_n^2(x)w_{_{JS}}(x)\frac{dx}{x},& &r_n^*(t):=\frac{t}{h_{n-1}}\int_{0}^1P_{n}(x)P_{n-1}(x)w_{_{JS}}(x)\frac{dx}{x}.
\end{align*}
According to (2.25) and (2.41) in \cite{CD10} which were obtained by using $(S_1)$ and $(S_2)$, we find that their expressions of $A_n$ and $B_n$ coincide with ours.
In addition, the differential relations (4.1) and (4.7) in their work are still valid for $\alpha,\beta>-1$. Therefore, the $\sigma$-form of Painlev\'{e} V and the second order difference equation satisfied by the logarithmic derivative of the associated Hankel determinant, which were established in \cite{CD10} and valid for $\alpha,\beta>0$, still hold for $\alpha,\beta>-1$. These results could also be deduced by substituting our expressions of $A_n$ and $B_n$ into $(S_1), (S_2)$ and $(S_2')$, and the derivation is likely to be more straightforward than that in \cite{CD10} where two additional auxiliary quantities were needed alongside $R_n^*$ and $r_n^*$.
\end{exm}

\begin{exm}\label{exJS-2}
The shifted Jacobi-type weight $w_{_{JS}}(x)=x^{\alpha}(1-x)^{\beta}(x-t)^{\gamma}, x\in[0,1],$ $\alpha,\beta>0, t<0,\gamma\in\mathbb{R}$ was considered in \cite{DZ10}. Supposing $\alpha,\beta>-1$, by inserting
\[\frac{(z-z^2)v_{_{JS}}'(z)-(x-x^2)v_{_{JS}}'(x)}{z-x}=\alpha +\beta +\gamma
   +\frac{\gamma\, t (1-t)
}{(z-t)(x-t)}
\]
into \eqref{An-JS}-\eqref{Bn-JS}, we readily get
\begin{align}
A_n(z)=&\frac{1}{z-z^2}\left[2n+1+\alpha+\beta+\gamma+\frac{t(1-t)}{z-t}a_n\right]\nonumber\\
=:&\frac{\mathcal{R}_n^*}{z}-\frac{\mathcal{R}_n}{z-1}+\frac{\mathcal{R}_n-\mathcal{R}_n^*}{z-t},\label{Anex13}\\
B_n(z)=&\frac{1}{z-z^2}\left[nz-n-\mathbf{p}(n)+\frac{t(1-t)}{z-t}b_n\right]\nonumber\\
=:&\frac{\texttt{r}_n^*}{z}-\frac{\texttt{r}_n}{z-1}+\frac{\texttt{r}_n-\texttt{r}_n^*-n}{z-t},\label{Bnex13}
\end{align}
where
\begin{align*}	 a_n(t):=&\frac{\gamma}{h_n}\int_0^1\frac{P_n^2(x)w_{_{JS}}(x)}{x-t}dx,\\
b_n(t):=&\frac{\gamma}{h_{n-1}}\int_0^1\frac{P_n(x)P_{n-1}(x)w_{_{JS}}(x)}{x-t}dx,
\end{align*}
and
\begin{align}	
\mathcal{R}_n^*(t):=&2n+1+\alpha+\beta+\gamma+(t-1)a_n,\label{R_n^*}\\
\mathcal{R}_n(t):=&2n+1+\alpha+\beta+\gamma+ta_n,\label{R_n}\\
\texttt{r}_n^*(t):=&-n-\mathbf{p}(n)+(t-1)b_n,\label{r_n^*}\\
\texttt{r}_n(t):=&tb_n-\mathbf{p}(n).\label{r_n}
\end{align}
Our expressions of $A_n$ and $B_n$ given by \eqref{Anex13}-\eqref{Bnex13} are identical in form to those in \cite{DZ10}. Moreover, from \eqref{r_n^*}-\eqref{r_n}, we readily see that
\begin{align}\label{diffRnrn}
b_n =\texttt{r}_n-\texttt{r}_n^*-n.
\end{align}
Inserting it into \eqref{r_n} gives us
\begin{align}\label{p-ex13}
\mathbf{p}(n)=(t-1)\texttt{r}_n-t\texttt{r}_n^*-nt.
\end{align}

In addition, by differentiating the orthogonality relations \eqref{or-Pn2-JS}-\eqref{or-Pnn-1-JS} w.r.t. $t$, we find
\[\frac{d}{dt}\ln h_n(t)=-a_n,\qquad\qquad \frac{d}{dt}\mathbf{p}(n,t)=b_n.\]
Hence, in view of \eqref{R_n} and \eqref{diffRnrn}, we have
\begin{align}
t\frac{d}{dt}\ln h_n(t)=&2n+1+\alpha+\beta+\gamma-\mathcal{R}_n,\label{Dh-ex13}\\ \frac{d}{dt}\mathbf{p}(n,t)=&\texttt{r}_n-\texttt{r}_n^*-n.\label{Dp-ex13}
\end{align}
Combining \eqref{Dh-ex13} with the fact that $D_n(t)=\prod_{j=0}^{n-1}h_j(t)$, and \eqref{Dp-ex13} with \eqref{p-ex13}, we are led to
\begin{align}
t\frac{d}{dt}\ln D_n(t)=&n(n+\alpha+\beta+\gamma)-\sum_{j=0}^{n-1}\mathcal{R}_j,\\
t\frac{d}{dt}\texttt{r}_n^*=&(t-1)\frac{d}{dt}\texttt{r}_n,
\end{align}
which are of the same form as (3.45) and (3.44) of \cite{DZ10} respectively.
Since the main results in \cite{DZ10} were established by using (3.44)-(3.45) and the expressions obtained by substituting $A_n$ and $B_n$ into $(S_1), (S_2)$ and $(S_2')$, we conclude that the $\sigma$-form of Painlev\'{e} VI deduced in \cite{DZ10} to characterize the associated Hankel determinant with $\alpha,\beta>0$ can be directly generalized to $\alpha,\beta>-1$.

\sloppy {To close this example, we point out that, when $\alpha,\beta>0$, our auxiliary quantities $\{\mathcal{R}_n^*,\mathcal{R}_n,\texttt{r}_n^*,\texttt{r}_n\}$ defined by \eqref{R_n^*}-\eqref{r_n} are identical to  $\{R_n^*, R_n,r_n^*,r_n\}$ given by (3.6)-(3.9) in \cite{DZ10}, namely,
\begin{align*}
\mathcal{R}_n^*=&\frac{\alpha}{h_n}\int_0^1P_n^2(x)w_{_{JS}}(x)\frac{dx}{x}=:R_n^*,\\
\mathcal{R}_n=&\frac{\beta}{h_n}\int_0^1P_n^2(x)w_{_{JS}}(x)\frac{dx}{1-x}=:R_n,\\
\texttt{r}_n^*=&\frac{\alpha}{h_{n-1}}\int_0^1P_n(x)P_{n-1}(x)w_{_{JS}}(x)\frac{dx}{x}=:r_n^*,\\
\texttt{r}_n=&\frac{\beta}{h_{n-1}}\int_0^1P_n(x)P_{n-1}(x)w_{_{JS}}(x)\frac{dx}{1-x}=:r_n.
\end{align*}
Indeed, when $\alpha,\beta>0$, through integration by parts, it was shown in \cite{DZ10} (c.f. (3.12)-(3.13)) that
\begin{align*}
a_n=&R_n-R_n^*,&
 b_n=&r_n-r_n^*-n.
\end{align*}
Inserting them into \eqref{R_n^*}-\eqref{r_n}, in view of (3.16) and (3.35) in \cite{DZ10} which were obtained from $(S_1)$ and $(S_2)$ respectively,  we come to the desired relations.}
\end{exm}

Following reasoning analogous to Section \ref{proof-J-RHP} or Section \ref{proof-J-1} with the aid of \eqref{lo-JS-1}, we establish the ladder operators for the monic shifted Jacobi-type orthogonal polynomials with several jump discontinuities or a Fisher-Hartwig singularity with simultaneous jump and root-type behavior.
	
\begin{theorem}
The monic orthogonal polynomials associated with
\begin{align}\label{wJS-Jumps}
w(x; \vec{t}\,) := w_{_{JS}}(x) \left( \omega_0 + \sum_{k=1}^{m} \omega_k \theta(x - t_k) \right),
\end{align}
where $w_{_{JS}}(x)$ is defined by \eqref{wJS} and independent of $\vec{t}=(t_1,\cdots,t_m)$ with $0\leq t_1<\cdots<t_m\leq1$, $\{\omega_k, 0\leq k\leq m\}$ and $\theta(\cdot)$ have the same meanings as in Theorem \ref{L-jump}, satisfy the ladder operators \eqref{loJS}-\eqref{roJS} with $A_n$ and $B_n$ reading
\begin{subequations}\label{AnBn-JS-3}
\begin{equation}
\begin{aligned}	 A_n(z)&=\frac{1}{z-z^2}\left[\frac{1}{h_n}\int_{0}^{1}\frac{(z-z^2)v_{_{JS}}'(z)-(x-x^2)v_{_{JS}}'(x)}{z-x}P_n^2(x;\vec{t}\,)w(x;\vec{t}\,)dx\right.\\
&\left.\qquad\qquad\quad+2n+1+\sum_{k=1}^{m}\frac{(t_k-t_k^2)\widehat {R}_{n,k}(\vec{t}\,)}{z-t_k}\right],\\
\end{aligned}
\end{equation}
\begin{equation}
\begin{aligned}
B_n(z)&=\frac{1}{z-z^2}\left[ \frac{1}{h_{n-1}}\int_{0}^{1}\frac{(z-z^2)v_{_{JS}}'(z)-(x-x^2)v_{_{JS}}'(x)}{z-x}P_n(x;\vec{t}\,)P_{n-1}(x;\vec{t}\,)w(x;\vec{t}\,)dx\right.\\
&\left.\qquad\qquad\quad+n(z-1)-\mathbf{p}(n)+\sum_{k=1}^{m}\frac{(t_k-t_k^2)\widehat{r}_{n,k}(\vec{t}\,)}{z-t_k}\right],
\end{aligned}
\end{equation}
\end{subequations}
where
\begin{equation*}
\widehat{R}_{n,k}(\vec{t}\,):=\frac{\omega_kw_{_{JS}}(t_k)P_n^2(t_k;\vec{t}\,)}{h_n},\qquad\qquad\widehat{r}_{n,k}(\vec{t}\,):=\frac{\omega_kw_{_{JS}}(t_k)P_n(t_k;\vec{t}\,)P_{n-1}(t_k;\vec{t}\,)}{h_{n-1}}.
\end{equation*}
\end{theorem}

\begin{exm}
When $m=1$ and $w_{_{JS}}(x)=x^{\alpha}(1-x)^{\beta}$ in \eqref{wJS-Jumps}, by substituting
\[\frac{(z-z^2)v_{_{JS}}'(z)-(x-x^2)v_{_{JS}}'(x)}{z-x}=\alpha+\beta\]
into \eqref{AnBn-JS-3}, on account of
\[\frac{t_1-t_1^2}{z(1-z)(z-t_1)}=\frac{1}{z-t_1}+\frac{t_1-1}{z}+\frac{t_1}{1-z},\]
we find
\begin{align*}
A_n(z)=&\frac{2n+1+\alpha+\beta+(t_1-1)\widehat{R}_{n,1}}{z}+\frac{2n+1+\alpha+\beta+t_1\widehat{R}_{n,1}}{1-z}+\frac{\widehat{R}_{n,1}}{z-t_1},\\
B_n(z)=&\frac{-n-\mathbf{p}(n)+(t_1-1)\widehat{r}_{n,1}}{z}+\frac{-\mathbf{p}(n)+t_1\widehat{r}_{n,1}}{1-z}+\frac{\widehat{r}_{n,1}}{z-t_1},
\end{align*}
where
\begin{equation*}
\widehat{R}_{n,1}(t_1):=\frac{\omega_1t_1^{\alpha}(1-t_1)^{\beta}P_n^2(t_1;t_1)}{h_n},\qquad\qquad\widehat{r}_{n,1}(t_1):=\frac{\omega_1t_1^{\alpha}(1-t_1)^{\beta}P_n(t_1;t_1)P_{n-1}(t_1;t_1)}{h_{n-1}}.
\end{equation*}
This case with $\alpha,\beta>0$ was studied in \cite{ChenZhang10}. Their expressions of $A_n$ and $B_n$, combined with $(45)$ and $(47)$ which were deduced in their work by using $(S_1)$ and $(S_2)$,  agree with ours. The differential relations (19) and (24) in \cite{ChenZhang10} still hold for $\alpha,\beta>-1$. Therefore, the $\sigma$-form of Painlev\'{e} VI established in \cite{ChenZhang10} for the logarithmic derivative of the associated Hankel determinant with $\alpha,\beta>0$ remains valid for $\alpha,\beta>-1$.
\end{exm}

\begin{theorem}
The monic shifted Jacobi-type polynomials orthogonal w.r.t. the weight function
\[w(x;t)=w_{_{JS}}(x)|x-t|^{\gamma}(A+B\theta(x-t)),\qquad x,t\in[0,1],\]
with $w_{_{JS}}(x)$ defined by \eqref{wJS} and independent of $t$, satisfy the ladder operators \eqref{loJS}-\eqref{roJS} with $A_n$ and $B_n$ given by
\begin{align*}		 A_n(z)=&\frac{1}{z(1-z)}\left[\frac{1}{h_n}\int_{0}^{1}\frac{(z-z^2)v_{_{JS}}'(z)-(x-x^2)v_{_{JS}}'(x)}{z-x}P_n^2(x;t)w(x;t)dx+2n+1\right.\nonumber\\
&\left.\qquad\qquad+\frac{\gamma}{h_n}\int_{0}^{1}\frac{x-x^2}{(z-x)(x-t)}P_n^2(x;t)w(x;t)dx\right],\nonumber\\		 B_n(z)=&\frac{1}{z(1-z)}\left[\frac{1}{h_{n-1}}\int_{0}^{1}\frac{(z-z^2)v_{_{JS}}'(z)-(x-x^2)v_{_{JS}}'(x)}{z-x}P_{n-1}(x;t)P_n(x;t)w(x;t)dx\right.\nonumber\\
&\left.\qquad\qquad+n(z-1)-\mathbf{p}(n)+\frac{\gamma}{h_{n-1}}\int_{0}^{1}\frac{x-x^2}{(z-x)(x-t)}P_{n-1}(x;t)P_n(x;t)w(x;t)dx\right],
\end{align*}
where $v_{_{JS}}(x):=-\ln w_{_{JS}}(x)$.
\end{theorem}

\begin{exm}\label{JS-FH}
For $w_{_{JS}}(x)=x^{\alpha}(1-x)^{\beta}$ with $\alpha,\beta>-1$ in the above weight function $w(x;t)$, by inserting
\[\frac{(z-z^2)v_{_{JS}}'(z)-(x-x^2)v_{_{JS}}'(x)}{z-x}=\alpha+\beta\]
into the expressions of $A_n$ and $B_n$, in view of
\[\frac{x-x^2}{\left(z-z^2\right)(z-x) (x-t) }=\frac{1}{1-z}\left(1+\frac{t}{x-t}\right)+\frac{1}{z}\left(1+\frac{t-1}{x-t}\right)+\frac{1}{(z-x)(x-t)},\]
we get
\begin{align}\label{AnBn-ex15}
A_n(z)=&\frac{\emph{x}_n}{z}-\frac{\mathcal{R}_n}{z-1}+a_n,&
B_n(z)=&\frac{\emph{y}_n}{z}-\frac{\texttt{r}_n}{z-1}+b_n,
\end{align}
where
\begin{align}
\emph{x}_n(t):=&2n+1+\alpha+\beta+\gamma+(t-1)u_n,&\;\emph{y}_n(t):=&-n-\mathbf{p}(n)+(t-1)v_n,\label{xy-ex15}\\
\mathcal{R}_n(t):=&2n+1+\alpha+\beta+\gamma+tu_n,&\;\texttt{r}_n(t):=&-\mathbf{p}(n)+tv_n,\label{Rr-ex15}\\
a_n(t):=&\frac{\gamma}{h_n}\int_0^{1}\frac{P_n^2(x;t)w(x;t)}{(z-x)(x-t)}dx,&\; b_n(t):=&\frac{\gamma}{h_{n-1}}\int_0^{1}\frac{P_n(x;t)P_{n-1}(x;t)}{(z-x)(x-t)}w(x;t)dx,\nonumber
\end{align}
and
\begin{align*}
u_n(t):=&\frac{\gamma}{h_n}\int_{0}^1\frac{P_n^2(x;t)}{x-t}w(x;t)dx,& v_n(t):=&\frac{\gamma}{h_{n-1}}\int_{0}^1\frac{P_n(x;t)P_{n-1}(x;t)}{x-t}w(x;t)dx.
\end{align*}
Our expressions \eqref{AnBn-ex15} have the same form as  (2.3)-(2.4) in \cite{MinChenMMAS}. Moreover, from the above definitions of $\emph{y}_n$ and $\texttt{r}_n$, it follows that
\begin{align}\label{vry}
v_n=\texttt{r}_n-\emph{y}_n-n.
\end{align}
Plugging it back into the definition of $\texttt{r}_n$ yields
\begin{align}\label{p-ex15}
\mathbf{p}(n)=(t-1)\texttt{r}_n-t\emph{y}_n-nt.
\end{align}
In addition, the differentiation of the orthogonality relations w.r.t. $t$ gives us
\[\frac{d}{dt}\ln h_n(t)=-u_n,\qquad\qquad\frac{d}{dt}\mathbf{p}(n,t)=v_n.\]
Hence, in light of the definition of $\mathcal{R}_n$ given in \eqref{Rr-ex15} together with \eqref{vry}, we have
\begin{align}
t\frac{d}{dt}\ln h_n(t)=&2n+1+\alpha+\beta+\gamma-\mathcal{R}_n,\label{Dh-ex15}\\
\frac{d}{dt}\mathbf{p}(n,t)=&\texttt{r}_n-\emph{y}_n-n.\label{Dp-ex15}
\end{align}
Combining \eqref{Dh-ex15} with the fact that $D_n(t)=\prod_{j=0}^{n-1}h_j(t)$, and \eqref{Dp-ex15} with \eqref{p-ex15}, we find
\begin{align}
t\frac{d}{dt}\ln D_n(t)=&n(n+\alpha+\beta+\gamma)-\sum_{j=0}^{n-1}\mathcal{R}_j,\\
(t-1)\texttt{r}_n'(t) =&t\emph{y}_n'(t),
\end{align}
which are of the same form as (3.3) and (3.6) of \cite{MinChenMMAS}.
Since the main results in \cite{MinChenMMAS} were deduced by using (3.3), (3.6) and the identities obtained by substituting $A_n$ and $B_n$ into $(S_1), (S_2)$ and $(S_2')$, we conclude that the $\sigma$-form of Painlev\'{e} VI established in  \cite{MinChenMMAS} (c.f. (3.19)) to characterize $t\frac{d}{dt}\ln D_n(t)$ with $\alpha,\beta>0$ can be directly generalized to $\alpha,\beta>-1$.

We finally remark that, when $\alpha,\beta>0$,
our auxiliary quantities $\{\emph{x}_n,\emph{y}_n,\mathcal{R}_n,\texttt{r}_n\}$ defined by \eqref{R_n^*}-\eqref{r_n} are essentially $\{x_n, y_n,R_n,r_n\}$ given in \cite{MinChenMMAS}, namely,
\begin{align*}
\emph{x}_n=&\frac{\alpha}{h_n}\int_0^1P_n^2(x)w_{_{JS}}(x)\frac{dx}{x}=:x_n,\\
\mathcal{R}_n=&\frac{\beta}{h_n}\int_0^1P_n^2(x)w_{_{JS}}(x)\frac{dx}{1-x}=:R_n,\\
\emph{y}_n=&\frac{\alpha}{h_{n-1}}\int_0^1P_n(x)P_{n-1}(x)w_{_{JS}}(x)\frac{dx}{x}=:y_n,\\
\texttt{r}_n=&\frac{\beta}{h_{n-1}}\int_0^1P_n(x)P_{n-1}(x)w_{_{JS}}(x)\frac{dx}{1-x}=:r_n.
\end{align*}
In fact, when $\alpha,\beta>0$, via integration by parts, it was shown in \cite{MinChenMMAS}  (c.f. (2.8)-(2.9)) that
\[u_n=R_n-x_n,\qquad\qquad v_n=r_n-y_n-n.\]
Inserting them into \eqref{xy-ex15}-\eqref{Rr-ex15}, on account of (2.10) and (2.18) in \cite{MinChenMMAS}  which were deduced from $(S_1)$ and $(S_2)$ respectively,  we arrive at the desired relations.
\end{exm}
		
\section{Conclusions}
\sloppy
In the existing literature, the ladder operators for the monic Laguerre-type, Jacobi-type and shifted Jacobi-type orthogonal polynomials were established under the parameter constraints $\lambda,\alpha,\beta>0$. They have been widely applied to study the recurrence coefficients of orthogonal polynomials, the partition functions, gap probabilities and linear statistics of Hermitian random matrices, etc. Although the obtained results are valid for $\lambda,\alpha,\beta>0$, the corresponding weight functions are actually well-defined for the broader range $\lambda,\alpha,\beta>-1$. This observation motivates our extension of the ladder operator formalism to this more general parameter space.
We develop two distinct approaches to construct the ladder operators: through the Riemann-Hilbert problem satisfied by the monic orthogonal polynomials; expressing $(1-z^2)P_n'(z)$ or $z(1-z)P_n'(z)$ as a linear combination of $\{P_k(z)\}_{k=0}^{n+1}$ for the Jacobi cases.

We validate our findings by investigating the monic Laguerre and Jacobi polynomials.
Moreover, we revisit several problems concerning the above-mentioned three types of weight functions, and generalize the results in prior work derived by using the ladder operator approach from $\lambda,\alpha,\beta>0$ to $\lambda,\alpha,\beta>-1$. To achieve this goal, we first compute $A_n$ and $B_n$ for $\lambda,\alpha,\beta>-1$ by using the definitions and verify that their expressions are consistent with those for $\lambda,\alpha,\beta>0$. Compared with prior literature, our calculation of $A_n$ and $B_n$ is typically more straightforward and sometimes their mathematical forms are also simpler, particularly regarding the monic Jacobi polynomials in Example \ref{classicJ}. Then, by differentiating the orthogonality relation w.r.t. the variable introduced in the weight function, we obtain differential equalities identical to the ones in previous studies for $\lambda,\alpha,\beta>0$. The above-mentioned two consistency guarantees ensure that the results deduced via the ladder operator approach, under the assumption that $\lambda,\alpha,\beta>0$, remain valid for $\lambda,\alpha,\beta>-1$.

In Examples \ref{ex-J3}, \ref{exJS-2} and \ref{JS-FH}, verifying the desired differential relations' consistency demands careful technique, particularly in Example \ref{ex-J3}. In addition, although in some examples (especially for the Jacobi cases), the definitions of our auxiliary quantities appear noticeably distinct from those in the existing literature, we show that they are actually identical when $\alpha,\beta>0$. See, e.g., Examples \ref{exJS-2} and \ref{JS-FH}.

\section*{Acknowledgements}
This work was supported by National Natural Science Foundation of China under grant numbers 12101343 and 12371257, and by Shandong Provincial Natural Science Foundation with project number ZR2021QA061.

\section*{Conflict of Interest Statement}
The authors declare that they have no known competing financial interests or personal relationships that could have appeared to influence the work reported in this paper.

\section*{Data availability statement}

Any data that support the findings of this study are included within the article.

\end{document}